\documentclass[12pt]{amsart}

\usepackage[utf8]{inputenc}
\usepackage{amsmath}
\usepackage{amsthm}
\usepackage{amssymb}
\usepackage{graphics}
\usepackage{graphicx}
\usepackage{amscd}
\usepackage{color}
\usepackage{todonotes}
\usepackage{bbm}
\usepackage{subcaption}
\usepackage{hyperref}
\usepackage{url}
\usepackage{thmtools} 
\usepackage{thm-restate}

\newtheorem{theorem}{Theorem}[section]
\theoremstyle{definition}
\newtheorem{definition}[theorem]{Definition}
\newtheorem{proposition}[theorem]{Proposition}
\newtheorem{lemma}[theorem]{Lemma}
\theoremstyle{definition}

\newtheorem{corollary}[theorem]{Corollary}
\newtheorem{question}[theorem]{Question}

\theoremstyle{remark}
\newtheorem{remark}[theorem]{Remark}
\theoremstyle{notation}
\newtheorem{notation}[theorem]{Notation}

\numberwithin{equation}{section}

\newcommand{\MCG}{{\mathcal MCG}}

\def\ZZ{{\mathbb Z}}

\renewcommand{\d}{\partial}
\newcommand{\s}{\sigma}

\title{Braids, fibered knots, and concordance questions}

\author[D.H.]{Diana Hubbard}
\address{Department of Mathematics, Brooklyn College, Brooklyn, NY 11210}
\email{diana.hubbard@brooklyn.cuny.edu}

\author[K.K.]{Keiko Kawamuro}
\address{Department of Mathematics, University of Iowa, Iowa City, Iowa 52242}
\email{keiko-kawamuro@uiowa.edu}

\author[F.C.K.]{Feride Ceren Kose}
\address{Department of Mathematics, UT Austin, Austin, TX 78712}
\email{fkose@math.utexas.edu}

\author[G.M.]{Gage Martin}
\address{Department of Mathematics, Boston College, Chestnut Hill, MA 02467}
\email{martaic@bc.edu}

\author[O.P.]{Olga Plamenevskaya}
\address{Department of Mathematics, Stony Brook University, Stony Brook, NY 11794}
\email{olga@math.stonybrook.edu}

\author[K.R.]{Katherine Raoux}
\address{Department of Mathematics, Michigan State University, East Lansing, MI 48824}
\email{raouxkat@msu.edu}

\author[L.T.]{Linh Truong}
\address{School of Mathematics, Institute for Advanced Study, Princeton, NJ 08540}
\email{ltruong@math.ias.edu}

\author[H.T.]{Hannah Turner}
\address{Department of Mathematics, UT Austin, Austin, TX 78712}
\email{hannahturner@math.utexas.edu}

\thanks{D.H. was supported in part by an AMS-Simons travel grant.  K.K. was partially supported by a Simons Foundation Collaboration Grant for Mathematicians. O.P. was partially supported by NSF grant DMS-1906260 and a Simons Fellowship.  L.T. was partially supported by NSF postdoctoral fellowship under grant DMS-1606451 and the Institute for Advanced Study. H.T. was supported by NSF graduate research fellowship under grant no. DGE-1610403.}

\begin{document}

\bibliographystyle{alpha}

\maketitle

\begin{abstract} Given a knot in $S^3$, one can associate to it a surface diffeomorphism in two different ways.  
First, an arbitrary knot in $S^{3}$ can be represented by braids, which can be thought of as diffeomorphisms of punctured disks. Second,
if the knot is fibered -- that is, if its complement fibers over $S^1$ --  one can consider the monodromy of the fibration. 
One can ask to what extent properties of these surface diffeomorphisms dictate topological properties of the corresponding knot.
In this article we collect observations, conjectures, and questions addressing this, from both the braid perspective and the 
fibered knot perspective. We particularly focus on exploring whether properties of the surface diffeomorphisms relate to four-dimensional topological properties of knots such as the slice genus.

\end{abstract}

\section{Introduction}

In this article, we will discuss a number of results, observations and questions concerning the relation between topological properties of knots and certain surface 
diffeomorphisms associated to these knots. The relevant surface diffeomorphisms arise in two ways. \begin{enumerate}
\item For a fibered knot $K \subset S^3$, we consider the fibration of the knot complement $S^3 \setminus K$ over $S^1$. The fiber surface and the monodromy of this fibration are uniquely determined by $K$ (see \cite[p.34]{EN85}, \cite[Corollary 8.3]{Sa12}). 
\item 
An arbitrary knot in $S^3$ can be (non-uniquely) represented as a braid closure, and we can study the corresponding braid monodromy. Recall that Artin's braid group $B_n$ on $n$-strands  \cite{Ar25} is
the mapping class group of an $n$-punctured disc $D_n$.
\end{enumerate}
Working from these different perspectives, we will explore how the fiber monodromy or braid monodromy is related to invariants of the underlying knot, such as the Seifert genus $g_3(K)$ and the slice genus $g_4(K)$. As usual, $g_3(K)$ is the smallest genus of a Seifert
surface the knot $K$ bounds in the 3-sphere, and $g_4(K)$ is the smallest genus
of a smooth oriented embedded surface that the knot $K$ bounds in the 4-ball. 

Intuitively, the more full twists that a braid $\beta$ contains, the more complicated its closure $\widehat{\beta} = K$ is. Let $\Delta^2 = (\sigma_1 \dots \sigma_{n-1})^n$ denote a full twist on $n$ strands. We expect the braid closure's genus $g_3(K)$ and slice genus $g_4(K)$ to grow larger as more full twists are added to the braid. This idea is made precise in the following proposition.


\begin{restatable}{proposition}{asymp} \label{asymp} Let $\beta \in B_m$ be a braid, and let $K_n$ be the
closure of the braid $\beta \Delta^{2n}$.  When $n$ becomes large, $g_3(K_n)=g_4(K_n)$ and this value grows roughly as $\frac{1}{2}
n m^2$. 
\end{restatable}
Proposition~\ref{asymp} gives us an asymptotic understanding of the behavior of the genus and slice genus as full twists are added. As for non-asymptotic behavior, for braids of three or more strands, we use Sato's work \cite{Sa18} to show concatenating full twists to a braid will change the concordance class of its closure. 
\begin{restatable}{proposition}{propsato} \label{prop:sato} Let $\beta \in B_{m}$ be a braid, $m \geq 3$, and $\beta_{k} = \beta\Delta^{2k}$, $k \in \mathbb{Z}, k > 0$. 
Then all braid closures  $\widehat\beta$,  $\widehat{\beta_k}$, $k >0$ lie in pairwise distinct concordance classes. 
\end{restatable}

We now focus on the boundary twisting of the monodromies. The amount of boundary twisting is quantified by the {\em fractional Dehn twist coefficient}, defined in Section~\ref{sec:FDTC}. The idea of the fractional Dehn twist coefficient, or FDTC,
first appeared in \cite{GO89} in the context of essential laminations. In the context of open books and contact topology, it was  developed and applied
 in \cite{HKM07}, \cite{HKM08}, and explored further by many authors, \cite{BE13, KR13, IK17, HM18}.  For classical braids, 
a similar notion (via a somewhat different approach) was studied in \cite{Ma04}. The fractional Dehn twist coefficient was also studied from the braid- and knot-theoretic perspective in \cite{Pl18,FH19, EV15} among others.
A generalization of the fractional Dehn twist coefficient to the case of braids in arbitrary open books, and a detailed proof that different definitions are equivalent, is given in \cite{IK19}. 

\begin{notation}\label{notation} 
We will use notation $FT(K)$ for the fractional Dehn twist coefficient
of the fibered monodromy for a fibered knot $K$ in $S^3$. We use $BT(\beta)$ for the fractional Dehn twist coefficient of a braid monodromy. 
\end{notation}

We will explore how the fractional Dehn twist coefficient of the braid monodromy is related to invariants of the underlying knot, such as the Seifert genus $g_3(K)$ and the slice genus $g_4(K)$. Given a braid representative $\beta$ whose closure is a knot $K = \widehat{\beta}$,  we seek lower bounds for the genus $g_3(K)$ and the slice genus $g_4(K)$ in terms of the fractional Dehn twist coefficient $BT(\beta)$. 

In \cite{It11} Ito used braid foliations to give a bound for the genus of a knot $K$ in terms of the fractional Dehn twist coefficient $BT(\beta)$ of its braid representative $\beta$. See Theorem \ref{thm:ito} for a more precise statement. Ito's theorem serves as motivation for our work. 

\begin{theorem}[\cite{It11}]\label{thm:Ito1}
 Let $\beta$ be a braid whose closure is isotopic to a knot $K \subset S^3$. Then 
$$
|BT(\beta)| \leq g_3(K)+2.
$$
\end{theorem}

For quasipositive braids, we provide a new upper bound on $BT(\beta)$ in terms of the slice genus in the following theorem.

\begin{theorem}\label{thm:quasipositiveBTbound} Let $n \geq 3$. 
Let $\beta\in B_n$ be a quasipositive $n$-braid whose braid closure $\hat\beta$ is a knot then $$BT(\beta)\leq 2g_4(\hat\beta)+(n-2).$$
\end{theorem}

The upper bound on $BT(\beta)$ is sharp by Corollary~\ref{slice-BT}, which constructs slice knots ($g_4(K) = 0$) represented by quasipositive $n$-braids $\beta$ with $BT(\beta) = n-2$.

The analogue of Ito's Theorem~\ref{thm:Ito1} with genus $g_3(K)$ replaced by slice genus $g_4(K)$ does not hold by Corollary~\ref{slice-BT} and Proposition~\ref{examplesthatkillthedream}. Notably, the non-quasipositive knots of Proposition~\ref{examplesthatkillthedream} satisfy the same inequality as Theorem~\ref{thm:quasipositiveBTbound}, leading us to ask whether such a bound holds in general.

\begin{restatable}{question}{quesgenusbounds}\label{conj:genusbounds} For any $n$-strand $\beta$ representing a knot $K$, is it always true that
$$
|BT(\beta)| \leq 2g_4(K)+n-2?
$$
If not, is there a bound of the order 
$$
|BT(\beta)| \leq C(n) g_4(K)+D(n),
$$
where $C(n), D(n)$ are constant for each fixed $n$?
\end{restatable}

As evidence towards a positive answer to Question~\ref{conj:genusbounds}, for $3$-braids representing knots of finite order in the concordance group, we provide an upper bound on $BT(\beta)$ in the following theorem.

\begin{restatable}{theorem}{thmslicethreebraids}\label{slice3-braids} Let  $K$ be a knot that can be represented as the closure of a braid in $B_{3}$. Suppose further that $K$ is slice, or more generally, that $K$ 
has a finite order in the concordance group. Then any $3$-braid representative $\beta$ of $K$ satisfies $|BT(\beta)| \leq 1$.
\end{restatable}

The proof of Theorem~\ref{slice3-braids} relies on Murasugi's classification of three-braids \cite{Mu74} and work of Baldwin \cite{Bal08} and is given in Section~\ref{sec:BTgenusbounds}. Section~\ref{sec:BTgenusbounds} concludes with a discussion on computational evidence for the bounds in Question~\ref{conj:genusbounds}.

One might hope to use the Ozsv\'ath-Szabo $\tau$-invariant \cite{OS03} or the Rasmussen $s$-invariant \cite{Ra10} to construct an upper bound on the fractional Dehn twist coefficient $BT$ of a braid. We show that this idea fails in Proposition~\ref{prop:BTtau}, by 
constructing a family of braids with small $\tau$-invariant and small $s$-invariant but large $BT$.

%
%
%
%

\begin{restatable}{proposition}{propBTtau}\label{prop:BTtau}
Let $\beta=(\Delta^2)^k\sigma_1^{-1}\sigma_2^{-(6k-1)}$ in $B_{3}$. Then $\hat\beta$ is a knot, 
$BT(\beta) \geq k-1$, $|\tau(\hat\beta)|\leq1$ and $|s(\hat\beta)|\leq2$.
\end{restatable}

It is important to note that $BT(\beta)$ is very sensitive to braid stabilizations. Recall that the positive braid stabilization of $\beta \in B_n$ is the braid $\beta_+=\beta \sigma_n \in B_{n+1}$, and the 
negative stabilization is the braid  $\beta_-=\beta \sigma_n^{-1} \in B_{n+1}$. The braid closures of $\beta$, $\beta_+$ and $\beta_-$  are all isotopic to the same link. 
Proposition~\ref{prop:stab} shows that stabilized braid representatives have fractional Dehn twist coefficient $BT(\beta_+)$ and $BT(\beta_-)$ that are bounded between $-1$ and $1$. Thus, a stabilized braid representative for a knot $K$ will not provide a useful  lower bound for the genus $g_3(K)$ or slice genus $g_4(K)$. To our knowledge it is an open question how to determine which braid representative has largest possible fractional Dehn twist coefficient.

Finally, we explore the relationship between the fractional Dehn twist coefficient $FT(K)$ of a fibered knot $K$ and other topological invariants of $K$, including its knot Floer stable equivalence class in Section~\ref{sec:knotfloerstable} and slice genus in Section~\ref{sec:fdtcsliceknots}.

\subsection*{Acknowledgments} This article resulted from work completed at the Women in Symplectic and Contact Geometry and Topology Workshop (WiSCon) hosted by ICERM. 
J. Elisenda Grigsby, Diana Hubbard, Keiko Kawamuro, and Olga Plamenevskaya were mentors 
for Project 6, titled ``Homological invariants, braids, transverse links, and surfaces". 
The participants were Elaina Aceves, Celeste Damiani, Feride Ceren Kose, Gage Martin, Juanita Pinz\'on Caicedo, Katherine Raoux, Linh Truong, and Hannah Turner. GM and LT thank Siddhi Krishna for helpful discussions. The results in this article reflect the observations of two sub-groups of Project 6. 
The mentors and participants of Project 6 would like to thank the organizers of WiSCon and are grateful
to ICERM and to the AWM for supporting WiSCon via the AWM ADVANCE grant NSF-HRD 1500481. Finally, the authors thank the anonymous referee for their feedback; in particular they suggested an improvement to Proposition \ref{asymp}.


\section{Background on fractional Dehn twist coefficient and braids}\label{sec:FDTCbraidorderings}
We begin with background on the braid group and the fractional Dehn twist coefficient and collect some previous results that inspired our work. 
\subsection{The braid group} \label{sec:braidgroup}

Recall that {Artin's braid group $B_n$ on $n$-strands}  \cite{Ar25} is
the mapping class group of an $n$-punctured disc $D_n$.  It admits the presentation
 $$B_{n} = \left\langle \sigma_{1}, \ldots, \sigma_{n-1} \Bigg\vert \begin{array}{lr} 
\sigma_{i}\sigma_{j} = \sigma_{j}\sigma_{i} : |i-j| \geq 2 \\
\sigma_{i}\sigma_{i+1}\sigma_{i} = \sigma_{i+1}\sigma_{i}\sigma_{i+1} : 1 \leq i \leq n-2
\end{array}
\right\rangle.$$
The $\sigma_i$'s are usually referred to as the {\em standard} or Artin generators of $B_n$.  
A \textit{n-braid} $\beta$ is an element of $B_n$.  The braid  $\beta$ can be considered as a proper embedding of $n$ strands in $D^2 \times I$, so that each generator $\sigma_i$ corresponds to 
a positive half-twist between the $i$-th and $(i+1)$-th strands. When  $D^2 \times \{0\}$ is glued to $D^2 \times \{1\}$ by the identity map, 
we obtain a link in $D^2 \times S^1 \subset S^3$ called the closure of $\beta$, denoted $\hat\beta$.   A classical theorem of Alexander \cite{Al23}  allows  to reverse this process (although not uniquely): every knot or link in $S^3$ 
can be represented as a closed braid.
The \textit{braid index} of a link is the minimum number $n$ such that there exists an $n$-braid $\beta$ whose closure represents the given link. 

Thus, a classical braid can be considered from several viewpoints: 1)
a braid is a ``word'' in standard generators $\sigma_i$ of the braid group
$B_n$;  2) the braid closure is a
knot or link that can be studied by means of classical knot theory and
by modern knot-homological invariants (such as knot Floer homology and Khovanov homology); 3) the braid
monodromy action  shows how the braid  twists, and is
related to ideas from geometry and dynamics of surface diffeomorphisms. There are many excellent resources for learning more about braids; for instance, see \cite{Bir05}, \cite{Bir16}, and \cite{EV15}. 

\subsection{Fractional Dehn twist coefficient}\label{sec:FDTC}

The amount of boundary twisting of a surface diffeomorphism is quantified by the {\em fractional Dehn twist coefficient} (FDTC), defined geometrically 
as follows. Let $S$ be  a  compact oriented surface  with connected boundary and $\chi(S)<0$, and  $\phi: S \to S$ a  homeomorphism fixing the boundary $\partial S$ pointwise.
Using the Nielsen-Thurston classification, we can find a free isotopy   $\Phi: S \times [0,1] \to S$ connecting $\phi$ to its Nielsen-Thurston representative, so that $\Phi_0=\phi$ and 
$\Phi_1$ is either periodic, reducible, or pseudo-Anosov. (Note that $\Phi_1$ no longer fixes the boundary of $S$.) 
The fractional Dehn twist coefficient $FT(\phi)$ is defined as the winding number of the arc  $\Phi (p \times [0,1])$ for a chosen basepoint $p \in \partial S$. It can be shown that $FT(\phi)$ 
depends only on the isotopy class of $\phi$ {\em rel} boundary, and is independent of the choice of basepoint $p$ and isotopy $\Phi$.
Moreover, $FT(\phi)$ is always a rational number. Although we primarily focus on knots in $S^3$, one could also consider the FDTC for fibered knots in an 
arbitrary closed oriented 3-manifold $Y$ (and for corresponding open books on $Y$). 

The FDTC of a braid is defined by considering the braid monodromy as a boundary-fixing homeomorphism of the punctured disk, and taking the FDTC at the 
boundary of the disk. For surfaces with multiple boundary components, the FDTC can be defined for each component 
(the corresponding number measures twisting around the given component).

For a fibered knot $K$ in $S^{3}$, there are different monodromies to consider: the monodromy of the fibration $S^3 \setminus K \to S^1$ as well as the monodromy of a braid representative 
$\beta$ of $K$. These monodromies live on different surfaces and have very different fractional Dehn twist coefficients. In particular, the FDTC of the fibration is fixed, while the FDTC of the braid monodromy depends on the choice of the braid representative. 
As in Notation~\ref{notation} we will use notation $FT(K)$ for the FDTC
of the fibration monodromy and $BT(\beta)$ for the FDTC of the braid monodromy.
An example of how these two values can be different comes from the torus knots $T_{p,q}$. Since these are fibered in $S^{3}$, by work of Gabai and Kazez-Roberts, $|FT(T_{p,q})| \leq \frac{1}{2}$ (see Theorem \ref{s3-fibered-thm} below). However, a braid representative of $T_{p,q}$ is the word $\beta_{p,q} = (\sigma_{1} \cdots \sigma_{p-1})^{q}$ in $B_{p}$, and $|BT(\beta_{p,q})| = \frac{q}{p}$, which may be a large quantity. In general, the quantity $BT(\beta)$ can also be defined more algebraically, using braid orderings and Dehornoy's floor function, see Section \ref{sec:braidorder}.

We collect a few useful properties of the FDTC in the next proposition. 

\begin{proposition}[\cite{Ma04, IK17}]\label{prop:FDTCprops}
Let $t_\d$ denote the positive Dehn twist about a boundary-parallel curve. Then for any two boundary-fixing homeomorphisms $\phi, \psi$ of $S$, we have
\begin{enumerate}
\item[a)] (Quasimorphism) $|FT(\phi \psi) - FT(\phi) - FT(\psi)| \leq 1$.
\item[b)] (Homogeneity) $FT(\phi^n) = n FT(\phi)$.
\item[c)] (Behaviour under full twists) $FT(t_\d \phi) = FT(\phi)+1$.
\item[d)] (Conjugacy invariant) $FT(\phi) = FT(\psi \phi \psi^{-1})$. 
\end{enumerate}
All of the above properties also hold for braids and their fractional Dehn twist coeficients $BT$. Additionally, if $\beta\in B_n$ is an $n$-braid, we have $BT(\beta) \in \{ \frac{p}{q}$, 
where  $p,q \in \ZZ, 1 \leq q \leq n$ \}.
\end{proposition}

A related notion is that of a {\em right-veering} surface homeomorphism. As above, let $S$ be a compact oriented surface  with connected 
boundary, and $\phi: S \to S$ a homeomorphism fixing the boundary $\partial S$ pointwise. 
Let $\alpha$, $\beta:[0,1]\rightarrow S$ be arbitrary properly embedded oriented arcs 
with $\alpha (0)=\beta (0)=x\in\partial S$. Isotope $\alpha$ and $\beta$ so that they intersect transversely
with the fewest possible number of intersections. We say that $\beta$ is to the right of
$\alpha$ if  either $(\dot{\beta}(0),\dot{\alpha}(0))$ define the orientation of $S$ at $x$ or $\beta=\alpha$. 
We now can define right-veering: we say $\phi$ is right-veering if for any choice of 
base point 
$x\in \partial S$
and properly embedded oriented arc $\alpha$ starting at $x$,
$\phi(\alpha)$ is to the right of $\alpha$. In a similar way, one defines left-veering maps. The maps that {\em do not} satisfy the corresponding properties 
are called non-right-veering resp. non-left-veering. Note that the identity map is both right-veering and left-veering (and this is the only isotopy class with both properties.)


Very roughly, one can think of the fractional Dehn twist
coefficient as being a measurement of ``how right-veering" a surface homeomorphism is.  Indeed, when $\phi$ is irreducible, the FDTC of $\phi$ is greater than zero if and only if $\phi$ is right-veering: 
see  \cite{HKM07, KR13}. This gives a way to estimate the FDTCs:

\begin{proposition} \label{arcs-FT} \cite[Corollary 2.6]{KR13} (1) If $\alpha$ is a properly embedded oriented arc starting on $\partial S$, then $\phi(\alpha)$ to the right of $\alpha$ implies $FT(\phi)\geq 0$, and 
$\phi(\alpha)$ to the left of $\alpha$ implies $FT(\phi)\leq 0$. 

\noindent (2) If $\beta$ is a braid with a braid word where a generator $\sigma_i$ enters with positive exponents only, then $BT(\beta)\geq 0$. If there is a braid 
word for $\beta$ where a generator $\sigma_j$ enters with negative exponents only, then $BT(\beta) \leq 0$.   
\end{proposition}

Part 2 of Proposition~\ref{arcs-FT} follows by finding corresponding arcs: after a conjugacy, we can assume that $\sigma_1$ enters with positive 
exponents only; then the obvious arc connecting the boundary of the disk to the first puncture is moved strictly to the right of itself. See also  
\cite[Corollary 5.5, Proposition 13.1]{Ma04} for related statements and alternate proofs for braids.

Proposition \ref{arcs-FT} gives an easy way to show that $FT(\phi)=0$ by finding two arcs, one moved to the right by $\phi$ and the other moved to the left. Similarly, 
$BT(\beta)=0$ if we can find a braid word for $\beta$ where one generator enters with positive exponents only, and the other with negative exponents only.

The right-veering property plays an important role in 3-dimensional contact topology:
a contact structure is tight if and only if the monodromy of each compatible open book is right-veering~\cite{HKM07}.
Moreover, a contact structure supported by an open book
with connected boundary and pseudo-Anosov monodromy with $FT\geq 1$ is isotopic to a
perturbation of a taut foliation, and therefore is weakly symplectically fillable and has non-vanishing Heegaard Floer contact invariant (with twisted coefficients) \cite{HKM08}.
For planar open books, the condition $FT> 1$ for every boundary component implies tightness of the supported contact structure \cite{IK15}. 
In a similar vein, transverse braids in the standard contact $S^3$ must be right-veering if their Floer- or Khovanov-homological transverse invariants are non-zero 
\cite{BG, BVV, Pl06, Pl18},
and braids with $BT>1$  have non-vanishing transverse hat-invariant in knot Floer homology \cite{Pl18}.  

A number of existing results connect the FDTC to knot invariants and topology of 3-manifolds. Gabai proved the following result concerning
the genus of fibered knots with pseudo-Anosov monodromy:

\begin{theorem}[\cite{Ga97}]\label{gabai-thm}
Let $K\subset S^3$ be a fibered knot. Suppose that its monodromy is either pseudo-Anosov or reducible with pseudo-Anosov near the boundary.
Then either $FT(K)=0$ or $FT(K) =1/r$, where $2 \leq \vert r \vert \leq 4(g_3(K)) - 2$. 
\end{theorem}

Kazez and Roberts in Corollary 4.3 of \cite{KR13} determined the possible $FT$ values for the periodic case and reducible case with the monodromy periodic near the boundary. Putting together Theorem \ref{gabai-thm} with their work yields a very strict bound on the values of $FT$ for fibered knots in $S^{3}$:

\begin{theorem}[{\cite[Theorem 4.5]{KR13}, \cite{Ga97}}]\label{s3-fibered-thm}
Let $K\subset S^3$ be a fibered knot. Then $FT(K) = 0$ or $FT(K) = 1/n$, where $n$ is an integer, $|n| \geq 2$. In particular, $-\frac{1}{2} \leq FT(K) \leq \frac{1}{2}$.
\end{theorem}

In a different direction, Hedden and Mark \cite{HM18} found an priori bound on the value of $FT$ for any fibered
knot in a fixed  3-manifold $Y$ in terms of the dimension of the Heegaard Floer homology with $\mathbb{F}=\ZZ/2$ coefficients and the size of the 
torsion in singular 
first homology. As a corollary, they get a bound for the FDTC of classical braids, via open books on the branched double cover. 
\begin{theorem}[\cite{HM18}]\label{HM:HF}  Let $Y$ be  a  closed  oriented  3-manifold.
Then for any fibered knot $K$ in $Y$ with monodromy $\phi$, the FDTC satisfies 

$$
|FT(\phi)| \leq \frac{1}2  (\dim_\mathbb{F} \widehat{HF}(Y) - |\emph{Tor} H_1(Y;\ZZ)|)+1. 
$$
\end{theorem}

\begin{corollary}[\cite{HM18}] \label{HM-braids} Let $\beta$ be an odd-strand braid representative of $K\subset S^3$, and let $\Sigma(K)$ denote the double-branched cover of $K$. Then 
$$|BT(\beta)|\leq \dim_\mathbb{F} \widehat{HF}(\Sigma(K)) -  |H_1(\Sigma(K))|+2.$$ 
\end{corollary}

Further, Hedden and Mark use the relation between the Khovanov homology of a link and the Heegaard Floer homology of its  branched double cover to prove a bound on $BT$ in terms of the 
rank of reduced Khovanov homology $\widetilde{Kh}$:
\begin{theorem}[\cite{HM18}]\label{HM:Kh}  
 Let $L$ be  a  link  in $S^3$,  and  let $\beta$ be any braid representative of $L$ with an odd number of strands. Then 
 $$
 BT(\beta) \leq  \dim_{\mathbb{F}} \widetilde{Kh} (-L) - |\det(L)|+2. 
 $$
\end{theorem}

\subsection{Dehornoy's braid ordering} \label{sec:braidorder}

We will now describe another, more algebraic, approach to the fractional Dehn coefficient for braids, based on Dehornoy's braid ordering.
The braid group $B_m$ is known to be orderable, namely, there exists a {\em
left-invariant} linear order on $B_m$, so that
 if  $\beta_2 \succ \beta_1$,  then  $\gamma \beta_2 \succ \gamma \beta_1$ for
any   $\gamma \in B_m$.
An ordering can be defined by
considering the action of the braid monodromy $\beta \in Map (D, Q)$ on the
  on the
``standard'' punctured disk $D$,  with the set of punctures $Q$ on the
$y$-axis, labeled $Q=\{p_1, p_2, \dots p_m\}$ from bottom to top. (See Figure~\ref{fig:arc-a} in Section 4.)
Roughly, $\beta_2 \succ \beta_1$ iff $\beta_2$ twists the $y$-axis
more to the right than $\beta_1$.  These ideas
can be traced back to W. Thurston and can be extended to 
obtain many different orderings (known as {\em Nielsen--Thurston}
orderings).  

Dehornoy \cite{De94} defined an ordering on $B_m$ from an
algebraic perspective. The algebraic definition is equivalent to the geometric one described above. We say that  
 $\beta  \succ 1$  iff the braid $\beta$ admits a
 braid word that contains the generator $\sigma_i$ but no
$\sigma_i^{-1}$ and no $\sigma_j$ for $j<i$. (A word of this form is called
$\sigma_i$-positive, which has the following geometric interpretation. Suppose we apply a $\sigma_{i}$-positive word $\beta$ to the punctured disk and pull the
image of the $y$-axis taut while fixing the punctures. Then the image of the $y$-axis under $\beta$
will first diverge from $y$-axis at a point between the $(i-1)$'st and $i$'th punctures and at this point, the image will go to the right of the $y$-axis.)
 Then, for $\beta, \beta'\in B_m$ we define
$\beta \succ \beta'$  if $(\beta')^{-1}\beta  \succ 1$.  
 From the algebraic perspective, checking
the basic properties of the ordering is highly non-trivial; from the  geometric perspective, it is not hard to see that
we get a well-defined linear order, \cite{FGRRW99}.
The algebraic  approach becomes useful if one wants to study
combinatorial braid invariants. For example, it follows from \cite{FGRRW99} that a non-right-veering braid is conjugate to a braid with a braid word where a generator 
$\sigma_i$ enters with negative exponents only (compare with Proposition~\ref{arcs-FT}). The specific property of the braid word allows to relate geometry and combinatorics; this approach was used in 
\cite{BG, Pl18}. 

Let $\Delta = (\sigma_1\sigma_2 \dots \sigma_{n-1})(\sigma_1 \sigma_2 \dots \sigma_{n-2})\dots (\sigma_1 \sigma_2)(\sigma_1) 
\in B_n$ be the Garside fundamental braid. Observe that $\Delta^2 = (\sigma_1 \dots \sigma_{n-1})^n$ is a full twist on $n$ strands. 

Using orderings, one defines the {\em Dehornoy's floor} $[\beta]_D=n
$ of a braid $\beta\in B_m$ as an integer $n$ such that 
$\Delta^{2n+2} \succ  \beta  \succ \Delta^{2n}$. 
The Dehornoy floor is related to the FDTC as follows
\cite{Ma04}:
\begin{equation} \label{BTfloor}
  [\beta]_{D} +1  \geq  BT(\beta) \geq [\beta]_{D} \quad
\text{ and  } \quad BT(\beta)  = \lim_{n \to \infty} \frac{[\beta^n
]_{D}}{n}.  
\end{equation}
In other words, the FDTC can be thought of as ``homogenization'' of the Dehornoy
floor, where one averages over large iterates of the braid. It is important to note that 
while $BT$ is an invariant of the conjugacy class of the braid, the Dehornoy floor is not. 

The following theorem of Ito \cite{It11} serves as motivation for our work. 

\begin{theorem}[\cite{It11}] \label{thm:ito}
If $K$ is represented by an $n$-strand braid $\beta$, then 
$$
|[\beta]_D|<\frac{4g_3(K)-2}{n+2}+\frac{3}{2}\leq g_3(K)+1.
$$
\end{theorem}

Note that Ito uses a slightly different definition of the Dehornoy floor, so that his formulas in \cite{It11} do not have the absolute value.
Ito's proof uses braid foliation techniques. It would be interesting to establish a similar bound via knot homologies.

\section{Concordance invariants and genus bounds}

In the last two decades, a number of knot-homological invariants were introduced to study knot concordance and give bounds for the slice genus. It would be interesting to find 
relations between these invariants and the FDTC of fibered knots or braids. We briefly review the invariants that we need.

A number of invariants come from knot Floer homology, introduced independently by Ozsv\'{a}th and Szab\'{o} in \cite{OS03} and by Rasmussen in \cite{Ra03}. See also \cite{Ma16} for a survey. 
For the simplest version, they associate a $\ZZ$-filtered 
chain complex $\widehat{CFK}(K)$ to a knot $K$. This chain complex is a powerful knot invariant; in particular, it detects the Seifert genus $g_3(K)$ 
\cite{OS04GenusBounds} and fiberedness \cite{Gh08},\cite{Ni07}. 
Total homology of this chain complex is of rank 1. The minimum filtration level in which the homology is supported yields an integer $\tau$, which is a concordance invariant,
\cite{OS03}. A concordance invariant $s(K)$ with similar properties was found by Rasmussen \cite{Ra10} using Khovanov homology \cite{Kh00}. For a link $L \subset S^3$, 
the Khovanov homology is  a link invariant $Kh(L)$, defined as  the cohomology of a 
bigraded chain complex $(CKh(D_L),d)$ associated to a diagram $D_L$ of the link. The invariant  $s(K)$ comes from Lee's deformation \cite{Le05} of the Khovanov differential and the 
resulting spectral sequence. The invariants $s(K)$ and $2\tau(K)$ share a number of properties, given in the next proposition, and coincide for many small knots but are known 
to differ in general 
\cite{HO08}.

\begin{proposition}[\cite{OS03, Ra10}]\label{prop:tauprops}  The maps $\tau: \mathcal{C} \rightarrow \mathbb{Z}$ and $s: \mathcal{C} \rightarrow \mathbb{Z}$ are surjective homomorphisms on the knot concordance group $\mathcal{C}$. They satisfy the following properties:
\begin{enumerate}
\item The absolute values of $\tau(K)$ and $\frac{s(K)}2$ give  lower bounds on the slice genus $g_4(K)$, 
$$
    |\tau(K)| \leq g_4(K), \qquad |s(K)| \leq 2 g_4(K).  
$$
\item If $\beta$ is a positive $n$-braid of length $k$ whose closure $\hat\beta$ is a knot, then $s(\hat \beta)= 2\tau(\hat\beta)=k-n+1.$ In particular, 
for the $(p,q)$-torus knot $T_{p,q}$, $p,q \geq 1$, we have 
$$
s(T_{p,q})= 2\tau(T_{p,q})= (p-1)(q-1)=2g(T_{p,q}).
$$
\item $\tau(-K)=-\tau(K)$  and $s(-K)=-s(K)$ where $-K$ denotes the concordance inverse of $K$. (In fact this follows directly from the fact that $\tau$ and $\frac{s}{2}$ are homomorphisms.) \\
\item Let $K_+$ be a knot, and $K_-$ the new knot obtained by changing one positive crossing in $K_+$ to a negative crossing. Then 
\[ \tau(K_+)-1 \leq \tau(K_-) \leq \tau(K_+)\]
    and
    \[ s(K_+)-2 \leq s(K_-) \leq s(K_+).\]

\end{enumerate}
    \end{proposition}

\begin{remark}
It is worth remarking that \emph{any} concordance homomorphism satisfying (1)-(3) will also satisfy a crossing change formula analogous to the one given in (4). Specifically, the argument given in Corollary 3 of \cite{Li04} uses only the formal properties (1)-(3).
\end{remark}

The bound on the slice genus is shown to be sharp for the classes of positive, quasipositive, and strongly quasipositive knots, 
 \cite{OS03, Pl04}, \cite{Li04}. In general, we have the following estimate for $\tau$ and $s$:

\begin{lemma}\label{tau-bound}
Let $\beta$ be an $n$-braid with $k$ positive crossings and $\ell$ negative crossings. If the closure of $\beta$ is a knot $\hat \beta$ then 
 $$ \frac12 (k-\ell-n+1) \leq \tau(\hat \beta)\leq \frac12(k-\ell+n-1)$$
and $$k-\ell-n+1 \leq s(\hat \beta)\leq k-\ell+n-1.$$

\end{lemma}
\begin{proof}
Let $\beta_+$ denote the braid obtained from $\beta$ by changing all of the negative crossings to positive and $\beta_-$ the braid 
obtained by changing all of the positive crossings to negative. Then by part (3) of Proposition~\ref{prop:tauprops}, we have $\tau(\hat\beta_+)-\ell\leq \tau(\hat\beta)$ 
and $\tau(\hat\beta)\leq \tau(\hat\beta_-)+k$.
Now  parts (2) and (4) imply that $$\tau(\hat\beta_+)=\frac{1}{2}(k+\ell-n+1)$$ since $\hat\beta_+$ is a positive braid, and
$$\tau(\hat\beta_-)=-\frac{1}{2}(k+\ell-n+1)$$ since $\hat\beta_-=-\hat\beta_+$.

Putting this information together, we see that $$ \frac{1}{2}(k-\ell-n+1)=\tau(\hat\beta_+)-\ell\leq \tau(\hat\beta)\leq  \tau(\hat\beta_-)+k= \frac{1}{2}(k-\ell+n-1).$$
Similarly, for $s$ we have $$s(\hat\beta_+)=k+\ell-n+1$$ and $$s(\hat\beta_-)=-s(\hat\beta_+)=-k-\ell+n-1.$$ The crossing change formula now gives the same inequality we had for $\tau$, except is it multiplied by the necessary factor of two, 
 $$s(\hat\beta_{+}) - 2\ell \leq s(\hat\beta)\leq s(\hat\beta_{-}) + 2k$$ 
Thus, $$k-\ell-n+1\leq s(\hat\beta)\leq k-\ell+n-1.$$
\end{proof} 

\begin{remark}
The lower bound for $\tau(\hat\beta)$ above appears in \cite[Corollary 11]{Li04}.
The lower bound for $s(\widehat{\beta})$  appears in~\cite[Proposition~4]{Pl06} and~\cite[Lemma~1.C]{Sh07}. Each of these bounds immediately
implies the slice-Bennequin inequality due to Rudolph. The upper bound for $s(\widehat{\beta})$ above is no stronger than the bounds from~\cite[Theorem~1.10]{Lo11} 
and additionally can be obtained from~\cite[Theorem~3.5]{Ma19}.
\end{remark}

The full knot Floer complex yields further concordance invariants, such as $\nu^+$ \cite{HW16}. We will discuss their connection to monodromies as well.   Since these invariants give lower bounds for the slice genus, we can ask about their relation to the FDTC.

Using properties of $\tau$ and $s$, we
can easily prove that if one starts  with a fixed knot and adds many full positive twists,  then 
indeed the genus, slice genus, $\tau$ and $s$ invariants all grow.  If $\Delta$ is the Garside element, the braid $\beta \Delta^{2n}$  is the concatenation of
$\beta$ with $n$ positive full 
twists. Note that $BT( \beta \Delta^{2n}  )= BT(\beta)+n$ by Proposition~\ref{prop:FDTCprops}.
We prove Proposition~\ref{asymp} from the Introduction:
\asymp*


\begin{proof} Suppose that $\beta$ has $k$ positive crossings and $l$ negative crossings. We first show that if $n\geq l$ then the braid $\beta\Delta^{2n}$ is a positive braid. Using the braid group relations it is not hard to show that $\Delta^2\sigma_i^{-1}$ is a positive braid for each $1\leq i\leq m-1$. Since $\Delta^2$ is central, each negative crossing can be removed by a single full-twist. 

Thus, when $n\geq l$, since $K_n$ is the closure of a positive braid, $g_3(K_n)=g_4(K_n)$ \cite{KM94}. In fact, we have $g_3(K_n)=g_4(K_n)=\tau(K_n)$ since $\tau$ is equal to the slice-genus for positive braids \cite{Pl04}. By Proposition \ref{prop:tauprops}(2), $\tau(K_n)$ is a simple function of length of the positive braid $\beta\Delta^{2n}$ and the number of strands:  $$g_3(K_n)=g_4(K_n)=\tau(K_n)=\frac12(k+(m-1)mn-l-m+1).$$

Since $m$, $l$ and $k$ are constant, we achieve the desired asymptotics. 

\end{proof}

The above proposition tells us that the FTDC, the 3-genus, and the slice genus
have similar asymptotics  when we add more and more full twists to a given braid,
but unlike Ito's bound, it gives no information about the relation between the FTDC and genus of
the original braid, before the twists are added. Note that Ito's bound is weaker if we add a large
number of positive twists to a fixed braid: it only says that the genus will
grow as  $\frac{1}{4}
n m$.

Adding a single full twist to a braid increases the FDTC by one. 
A natural question to ask is whether this move also necessarily changes the slice genus, or whether it changes the concordance class of the knot.
Certainly the above proposition shows that this is true asymptotically, but we show in the following proposition that 
Theorem 1.6 of Sato's work in \cite{Sa18} implies that this is also true for any fixed number of full twists for braids with three or more strands. 

\propsato*
\begin{proof}  Hom and Wu in \cite{HW16}  
define a refinement of the $\tau$-invariant called $\nu^+$ arising from the knot Floer complex. 
This invariant is a knot concordance invariant. In Proposition 1.5 of \cite{Sa18}, 
Sato defines a partial order $<_{\nu^{+}}$ on the concordance group mod $\nu^+$-equivalence. 
Let $\beta_{1}$ be the braid obtained by adding a single positive full-twist to $\beta$, that is, $
\beta_{1} = \beta\Delta^{2}$. Theorem 1.6 (2) of \cite{Sa18}  implies that $[\widehat{\beta}] <_{\nu^{+}} [\widehat{\beta_{1}}]$. 
In particular, $\widehat{\beta}$ and $\widehat{\beta_{1}}$ are not concordant. 
As we add more full twists, we see that
$[\widehat{\beta}] <_{\nu^{+}} [\widehat{\beta_{1}}] <_{\nu^{+}} [\widehat{\beta_{2}}] <_{\nu^{+}} \cdots <_{\nu^{+}}  [\widehat{\beta_{k}}]$.
So all of these knots live in different $\nu^{+}$-equivalence classes and none of them can be concordant to each other.
\end{proof}

We note that for braids with two strands, Proposition~\ref{prop:sato} is false, as the braids $\sigma_{1}^{-1}$ and $\sigma_{1}$ differ from each other 
by a single full twist in $B_{2}$ but both close to the unknot.

%
%
%
%

We prove Proposition~\ref{prop:BTtau} from the Introduction. The argument is based on the formal properties shared by $\tau$ and $s$.

\propBTtau*
\begin{proof} Observe that if $\Delta^2\beta'$ is any braid with $\hat\beta'$ a knot then $\widehat{\Delta^2\beta'}$ is also a knot. Now one can check that $\beta'=\sigma_1^{-1}\sigma_2^{-(6k-1)}$ is a knot.

Lemma \ref{tau-bound} implies $$-1=\frac{1}{2}(6k-(6k-1)-1-2)\leq \tau(\hat\beta)\leq \frac{1}{2}(6k-(6k-1)-1+2)=1.$$ Thus $|\tau(\hat\beta)|\leq 1$.  Similarly Lemma \ref{tau-bound} implies that $|s(\hat\beta)|\leq2$. To show that $FDTC(\beta) \geq k-1$, first note that using braid relations we can rewrite 
$$\beta = (\Delta^2)^{k-1} \sigma_{1}\sigma_{2}\sigma_{1}^{2} \sigma_{2}^{-6k+2}$$
For brevity let us write $\sigma_{1}\sigma_{2}\sigma_{1}^{2} \sigma_{2}^{-6k+2} = \alpha$. Note also that by property (c) of Proposition \ref{prop:FDTCprops}, $FDTC(\beta) = k-1 + FDTC(\alpha)$.
Finally, by Proposition~\ref{arcs-FT}, $BT(\alpha) \geq 0$ since it contains copies of $\sigma_{1}$ but no copies of $\sigma_{1}^{-1}$. 
\end{proof}

Thus any bound relating the fractional Dehn twist coefficient $BT(\beta)$ and the slice genus will have to resort to different proof methods than directly using the $\tau$ or the $s$-invariant.

\section{Quasipositive braids and the FDTC bounds}

A naive question would be to ask whether Theorem \ref{thm:ito} holds as stated if we replace the three-genus by the slice genus. 
For instance, slice genus and three-genus are equal for closures of positive braids \cite{KM94}, and so the slice genus version of Theorem \ref{thm:ito} 
immediately holds for positive braid closures.   The next natural class of braids to consider are quasipositive braids: 
braids that can be written as a product of conjugates of the positive Artin generators. 
For quasipositive braids, we will prove the Ito-like bound stated in Theorem~\ref{thm:quasipositiveBTbound}. In this section we assume the braid index $n \geq 3$.

%


\begin{definition}
We say that an $n$-braid $\beta \in B_n$ is {\em quasipositive} if it is represented by a braid word of the form 
$$(w_1 \s_{i_1} w_1^{-1})\ (w_2 \s_{i_2} w_2^{-1})\cdots (w_m \s_{i_m}w_m^{-1})$$
where $\s_{i_1},\dots,\s_{i_m} \in \{\s_1,\dots,\s_{n-1}\}$ and $w_i$ are some braid words in $\{\s_1^\pm,\dots,\s_{n-1}^\pm\}$. 
We say that the braid word has {\em quasi-positive-length} (qp-length) $m$. 
\end{definition}

Let us identify $D_n$ with the unit disk in $\mathbb R^2$ equipped with $(x, y)$ coordinates. Place $n$ punctures on the $y$-axis and call them $p_1, \dots, p_n$ so that the $y$-coordinate of $p_i$ is less than that of $p_{i+1}$ (see Figure~\ref{fig:arc-a}). 
Let $\rho_i$ be the subarc of the $y$-axis joining $p_i$ and $p_{i+1}$. 
\begin{figure}[htbp]
\includegraphics*[width=50mm,bb=50 100 320 350]{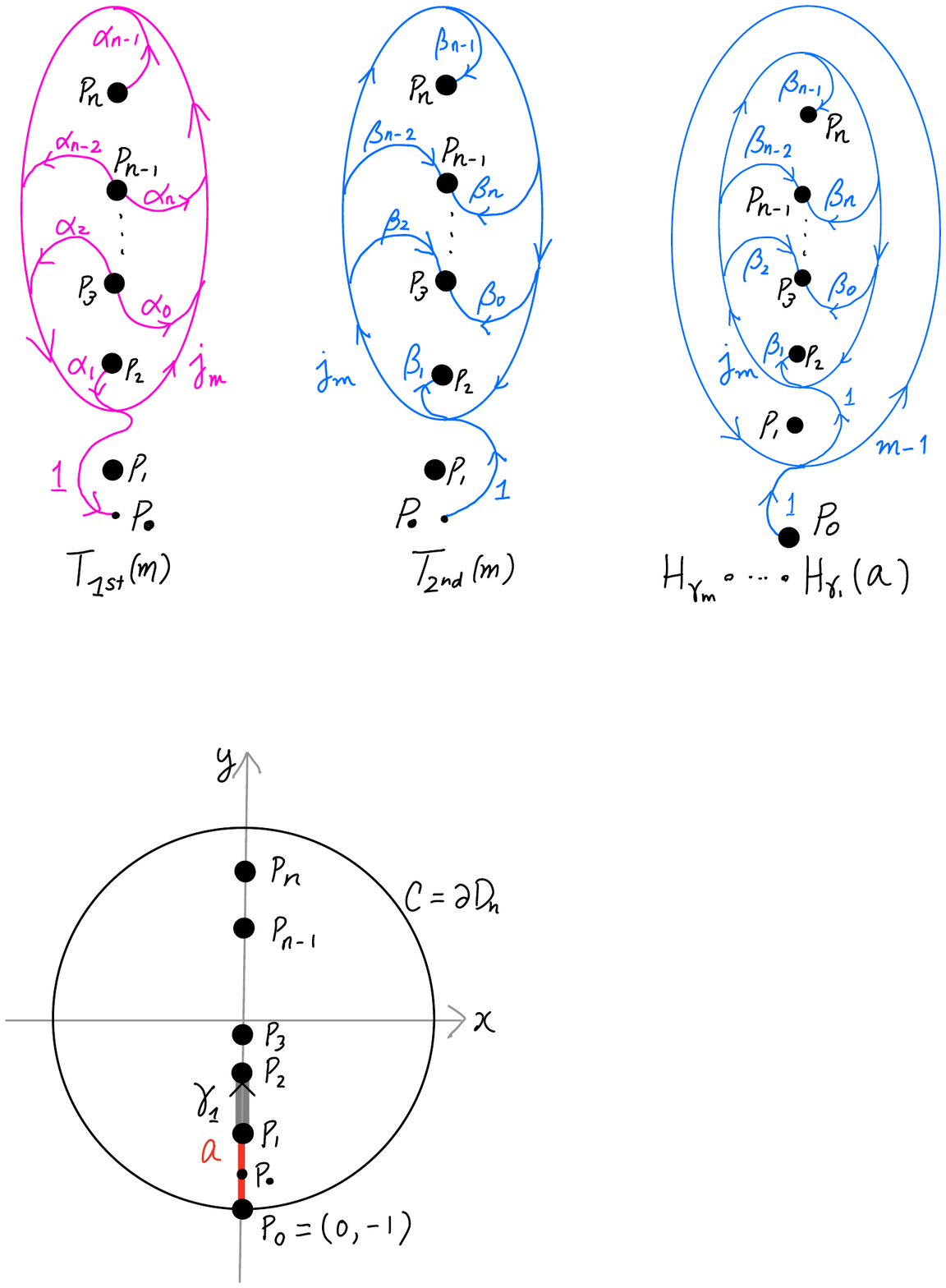}
\caption{Punctured disk $D_n$. Arcs $a$, point $p_\bullet$ and $\gamma_1:=\rho_1$ are used in the proof of Proposition~\ref{prop:m-2}.} 
\label{fig:arc-a}
\end{figure}

For a properly embedded arc $\gamma$ in $D_n$ joining two puncture points, let $H_\gamma \in \MCG(D_n)$ denote the positive half twist along the arc $\gamma$. 
The braid group $B_n$ is isomorphic to the mapping class group $\MCG(D_n)$. 
With this isomorphism $\phi:B_n\to\MCG(D_n)$, a braid word $w \in B_n$ is identified with the mapping class $\phi(w)$ denoted by $\phi_w := \phi(w) \in \MCG(D_n)$. 
Let $$\gamma_{w \s_i w^{-1}} := \phi_w (\rho_i)$$ be a properly embedded arc in $D_n$ that joins two distinct points of $p_1,\dots,p_n$.
Then the braid word $w \s_i w^{-1} \in B_n$ is identified with the positive half twist along the arc $\gamma_{w \s_i w^{-1}}$, thus, $$H_{\gamma_{w \s_i w^{-1}}}=\phi_{w \s_i w^{-1}}.$$ For more on this construction, see for instance \cite{EV15} Example 3.5 or \cite{Hay17} Lemma 2.4. 

\begin{proposition}\label{prop1}
Let $\beta$ be a quasipositive $n$-braid. Let $m\geq 1$. 
The braid $\beta$ has qp-length $m$ if and only if there exist properly embedded arcs $\gamma_1,\dots,\gamma_m$ joining the punctures in $D_n$ as above,
such that $\phi_\beta = H_{\gamma_m} \circ \cdots \circ H_{\gamma_1}.$
\end{proposition}

With Proposition~\ref{prop1} we may identify the braid $\beta \in B_n$ and the mapping class $\phi_\beta = H_{\gamma_m} \circ \cdots \circ H_{\gamma_1} \in \MCG(D_n)$ and we have $$BT(\beta) = FT(H_{\gamma_m} \circ \cdots \circ H_{\gamma_1}).$$

Here, we recall a useful lemma. 
The lemma applies for the FDTC with respect to any fixed boundary component $C$ of a surface $S$ ($S$ may have 
one or several boundary components).

\begin{lemma}[\cite{IK17}]
\label{lemma:fracDehn}
Let $\phi \in  \MCG(S)$. Fix a boundary component $C$ of the surface $S$ and let  $FT(\phi)$ stand
for the FTDC with respect to  $C$.   Let $T_C$ denote the positive Dehn twist about $C$.

If there exists an essential arc $\gamma \subset S$ that starts on $C$ and satisfies 
$T_{C}^{m}(\gamma) \preceq \phi(\gamma) \preceq T_{C}^{M}(\gamma)$ 
for some $m, M \in \ZZ$ then 
$$
m \leq FT(\phi) \leq M.
$$
\end{lemma}

In the above lemma, the symbol $\prec$ represents an ordering on the set of properly embedded arcs, see \cite[Definition 3.2]{IK19}.
We write $\alpha\prec\beta$ if arcs $\alpha$ and $\beta$ start at the same boundary point, realize the geometric intersection number, 
and $\beta$ lies on the right side of $\alpha$ in a small neighborhood of the starting point. (The notion ``$\beta$ to the right of $\alpha$'' was already discussed in Section~\ref{sec:FDTC}. Here, we write $\preceq$ to allow for the case $\alpha=\beta$.)

\begin{lemma} \label{lem:FTbound}
Let the braid index $n\geq 3$. 
Let $m\geq 1$. 
Let $\gamma_1,\dots,\gamma_m$ be properly embedded arcs in $D_n$ connecting two distinct  punctures. Then we have $FT(H_{\gamma_1})=0$
and $$0\leq FT(H_{\gamma_m} \circ \cdots \circ H_{\gamma_1}) \leq m-1.$$ 
\end{lemma}

\begin{proof} 
Since $n \geq 3$ there exists some essential arc that is fixed by the half twist $H_{\gamma_1}$. By Lemma~\ref{lemma:fracDehn} this means that $FT(H_{\gamma_1})=0$. 
(Note that if $n=2$ then $FT(H_{\gamma_1})=\frac{1}{2}$.) 
By the quasimorphism property of the FDTC in Proposition \ref{prop:FDTCprops} and induction on $m$, we obtain $0\leq FT(H_{\gamma_m} \circ \cdots \circ H_{\gamma_1}) \leq m-1.$
\end{proof}

Let $K$ be a link in $S^3 = \partial B^4$. Let $\chi_4(K)$ denote the maximal Euler characteristic of an oriented surface that is smoothly embedded in the 4-ball $B^4$ and  bounded by $K$.

\begin{lemma}[\cite{Ru93}] \label{chi-sharp}
Assume that $\beta \in B_n$ is a quasipositive braid of qp-length $m$. 
Then $\chi_4(\hat\beta) = n-m$. 
\end{lemma}

With the above two lemmas, we prove the following theorem. 
\begin{theorem}[cf. Theorem~\ref{thm:quasipositiveBTbound}]\label{thm:estimate}
Let $\beta\in B_n$ be a quasipositive $n$-braid of qp-length $m$. 
Then $$BT(\beta) \leq m-1= n- \chi_4(\hat\beta) -1.$$ 
In particular, when the braid closure $\hat\beta$ is a knot then $$BT(\beta)\leq 2g_4(\hat\beta)+(n-2).$$
\end{theorem}
\begin{proof}
By Lemma~\ref{lem:FTbound}, we have $BT(\beta) \leq m-1$. By Lemma~\ref{chi-sharp}, we have $m - 1 = n - \chi_4(\hat\beta) - 1$. 
\end{proof}

The next proposition shows that the 
 upper bound of  $m-1$ for $BT(\beta)$ is the best possible, as it guarantees the existence of a quasipositive braid $\beta$ of quasi-positive length $m$ with 
$BT=m-1$.  

\begin{proposition}\label{prop:m-2}
For every $m=2, 3, 4, \cdots$ there exist properly embedded arcs $\gamma_1, \cdots, \gamma_m$ in $D_n$ joining distinct puncture points such that
$$FT(H_{\gamma_m} \circ \cdots \circ H_{\gamma_1}) = m-1.$$
\end{proposition}

In the proof of Proposition~\ref{prop:m-2} we use oriented train tracks. 
A train track is a graph.  The edges are oriented and weighted. 
Each vertex has valence 3 or 4 where incoming edges and outgoing edges tangentially meet.  
As shown in the left picture in Figure~\ref{fig:tt-basic}, at the vertex two outgoing edges (with weights $b$ and $c$) are tangent to the incoming edge (with weight $a$) and the weights satisfy the equation $a=b+c$. (The same rule applies for a valence 4 vertex.) 
Replacing each edge of weight $a$ with $a$ parallel edges we obtain arcs. The right picture of Figure~\ref{fig:tt-basic} shows arcs 
carried by the train track with weights $a=6, b=4$ and $c=2$. 
Orientations of the edges in the train track induce orientations of the carried arcs. 

\begin{figure}[htbp]
\includegraphics*[width=80mm,bb=50 340 450 490]{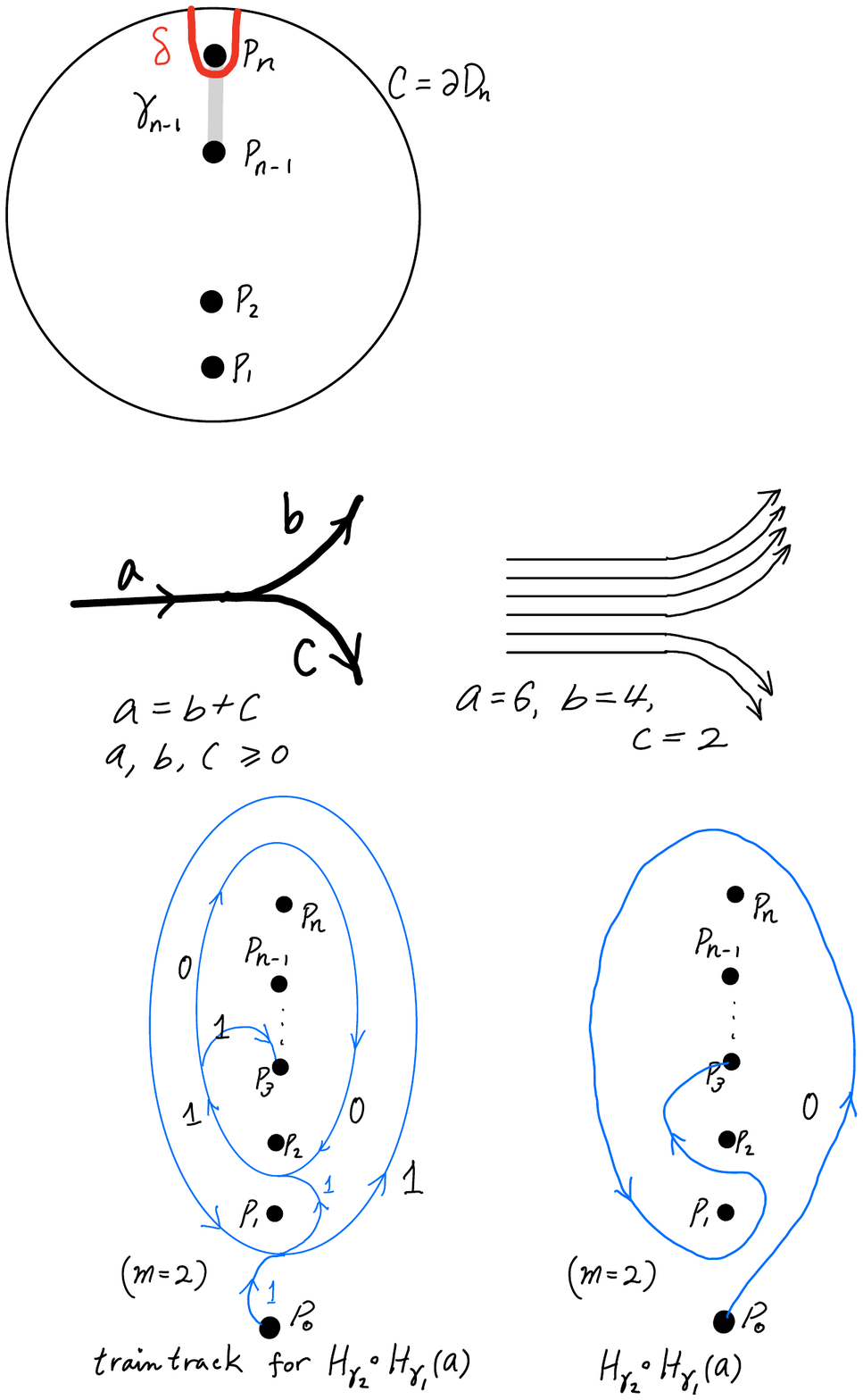}
\caption{} 
\label{fig:tt-basic}
\end{figure} 

\begin{proof}[Proof of Proposition~\ref{prop:m-2}]
Let $p_1,\cdots, p_n$ be the punctures of the disk $D_n$. 
Let $C=\partial D_n$ denote the boundary of the disk. 
Let $a$ be an arc connecting $p_1$ to $C$ as in Figure~\ref{fig:arc-a}. Let $\gamma_1:=\rho_1$,
the line segment of the $y$-axis joining $p_1$ and $p_2$. We orient $\gamma_1$ upward (from $p_1$ to $p_2$).

To define a sequence of arcs $\{\gamma_m| m=2, 3, \cdots\}$ we use two oriented train track templates $T_{\tt 1st}(m)$ and $T_{\tt 2nd}(m)$ depicted in Figure~\ref{fig:templates}. 
\begin{figure}[htbp]
\includegraphics*[width=130mm,bb=50 400 550 700]{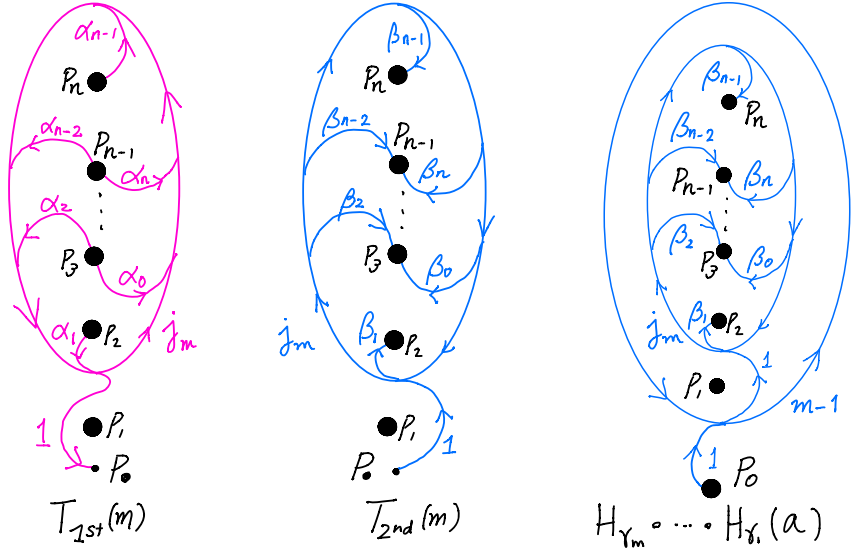}
\caption{Train tracks $T_{\tt 1st}(m)$ and $T_{\tt 2nd}(m)$ for $\gamma_m$, and a train track carrying the arc $H_{\gamma_m}\circ\cdots\circ H_{\gamma_1}(a)$.} 
\label{fig:templates}
\end{figure} 

In the train track $T_{\tt 1st}(m)$, the orientation of edges near punctures are all outward. The bottom edge has weight 1 and ends at a point $p_\bullet$ on the $y$-axis between $p_0$ and $p_1$. Labels $j_m \in \mathbb N$ and $
\alpha_{m,0}, \alpha_{m,1}$,$\dots$,$\alpha_{m,2n-5}$ $\in \{0,1\}$ denote the weights of the edges. 
The edges $\alpha_{m,i}$ are labeled clockwise and $\alpha_{m,0}$ starts from the puncture $p_3$. (The same rule applies for the edges $\beta_{m,i}$ mentioned below.)
For simplicity, the subscript $m$ will be omitted in the following and $\alpha_i = \alpha_{m,i}$.  
We require only one of $\alpha_0$,$\dots$,$\alpha_{2n-5}$ is $1$ and the rest of them are all $0$. The weights of the unlabeled edges can be easily computed. 

The edges of $T_{\tt 2nd}(m)$ are oriented and at each puncture the orientation points into the puncture. 
The weights $\beta_0, \cdots, \beta_{2n-5}$ are either $0$ or $1$ and only one of them is $1$. 
The starting point of the bottom edge is $p_\bullet$.

Let $m=2$. 
The arc $\gamma_2$ is defined to be an arc carried by the train track $T_{\tt 1st}(2)$ first, and then $T_{\tt 2nd}(2)$, where 
$j_2=1$ and $\alpha_1 = \beta_2=1$. 
We see that $\gamma_2$ starts at $p_2$ and ends at $p_3$ as in Figure~\ref{fig:gamma2}. 
\begin{figure}[htbp]
\includegraphics*[width=120mm,bb=20 170 553 720]{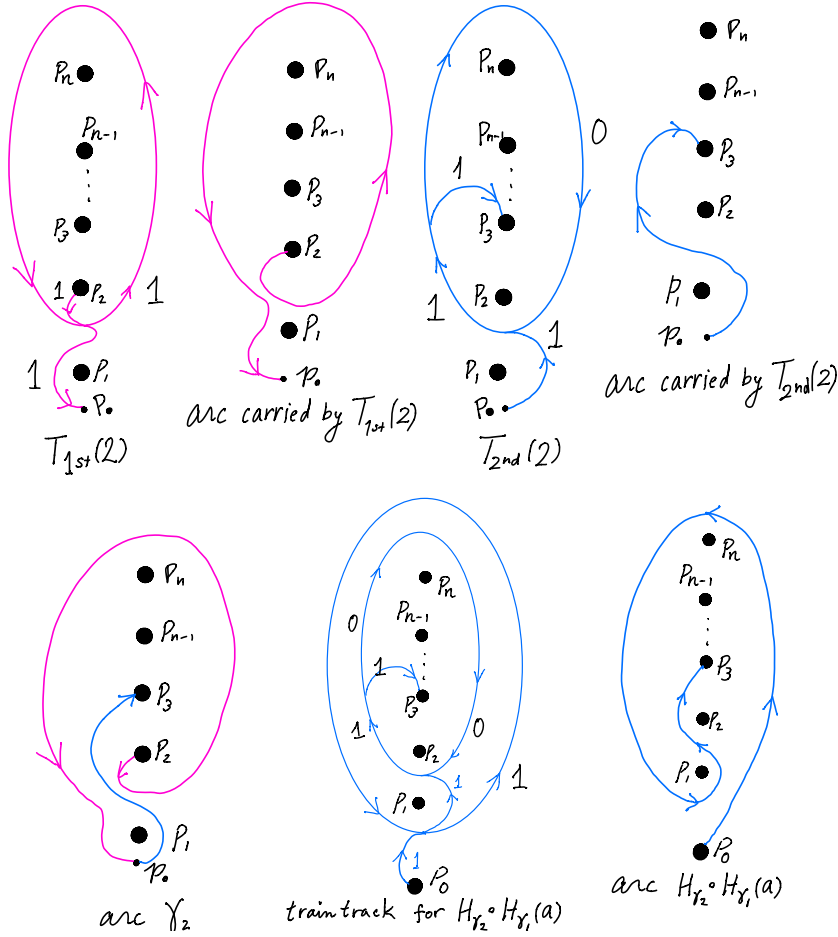}
\caption{($m=2$). The arcs $\gamma_2$ and $H_{\gamma_2}\circ H_{\gamma_1}(a)$ and their train tracks.}
\label{fig:gamma2}
\end{figure} 
The image of the arc $a$ under the diffeomorphism $H_{\gamma_2}\circ H_{\gamma_1}(a)$ is depicted in the bottom right picture in Figure~\ref{fig:gamma2}, which shows that $$T_C(a) \prec H_{\gamma_2}\circ H_{\gamma_1}(a),$$ 
thus $$FT(H_{\gamma_2}\circ H_{\gamma_1})\geq 1.$$

For a general $m$, 
suppose that $m-1 = k(2n-4) + r$ where $k \in \mathbb Z_{\geq 0}$ and $r \in \{0,1,\dots, 2n-5\}$ the remainder.  
The arc $\gamma_m$ is defined to be an arc that is carried by $T_{\tt 1st}(m)$ first and then $T_{\tt 2nd}(m)$ with $\alpha_r = \beta_{r+1}=1$ (here,
the subscript $r+1$ is considered  modulo $2n-4$) and 
$$
j_m= \left\lceil \frac{m-1}{2n-4} \right\rceil (= k \mbox{ or } k+1).
$$

For example, when $m=3$ the arcs $\gamma_3$ and $H_{\gamma_3}\circ H_{\gamma_2}\circ H_{\gamma_1}(a)$ are illustrated in Figure~\ref{fig:m=3}. 
\begin{figure}[htbp]
\includegraphics*[width=120mm,bb=50 120 540 690]{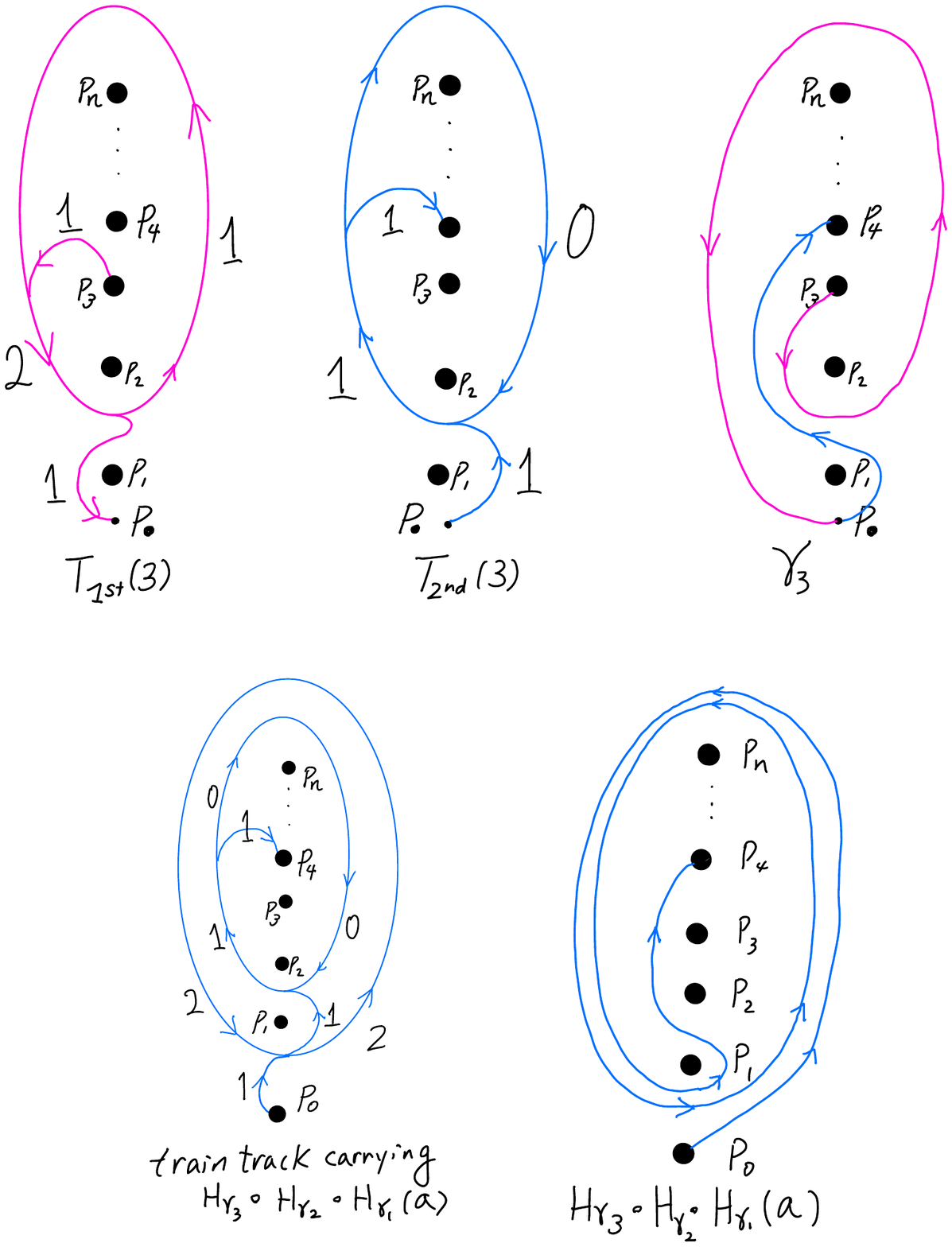}
\caption{($m=3$). Arcs $\gamma_3$ and $H_{\gamma_3}\circ H_{\gamma_2}\circ H_{\gamma_1}(a)$ and their train tracks.} 
\label{fig:m=3}
\end{figure}

See Figure~\ref{fig:templates} again. 
We observe two things: 
(1) Since $\alpha_r = \beta_{r+1}=1$ the templates $T_{\tt 1st}(m+1)$ and $T_{\tt 2nd}(m)$ are exactly the same except that they are oppositely oriented and $T_{\tt 1st}(m+1)$ goes to the left side of $p_1$ whereas $T_{\tt 2nd}(m)$ goes to the right side of $p_1$. 
(2) Observe that $T_{\tt 2nd}(m)$ is embedded in $H_{\gamma_m} \circ\cdots\circ H_{\gamma_1}(a)$. 

With these observations in mind, applying the half-twist $H_{\gamma_{m+1}}$ to the arc $H_{\gamma_m} \circ\cdots\circ H_{\gamma_1}(a)$, 
the embedded part $T_{\tt 2nd}(m)$ is erased by $T_{\tt 1st}(m+1)$ (which is the 1st part of $\gamma_{m+1}$) and then $T_{\tt 2nd}(m+1)$ (which is the 2nd part of $\gamma_{m+1}$) overwrites the erased part of $H_{\gamma_m} \circ\cdots\circ H_{\gamma_1}(a)$. The resulting arc is exactly $H_{\gamma_{m+1}} \circ\cdots\circ H_{\gamma_1}(a)$.

Therefore, by induction on $m$ we see that the image $H_{\gamma_m}\circ\cdots\circ H_{\gamma_1}(a)$ is carried by the train track in the right sketch in Figure~\ref{fig:templates}.

The outer circle in the right picture in Figure~\ref{fig:templates} is weighted $m-1$, which means that 
$$T^{m-1}_C (a) \prec H_{\gamma_m} \circ\cdots\circ H_{\gamma_1} (a).$$
By Theorem~\ref{thm:estimate} and Lemma~\ref{lemma:fracDehn} we obtain $FT(H_{\gamma_m} \circ\cdots\circ H_{\gamma_1} )=m-1$.  
\end{proof}

\begin{corollary} \label{slice-BT} There exist slice knots represented by quasipositive $n$-braids with $BT=n-2$. 
\end{corollary}
\begin{proof} For $m=n-1$, the construction of Proposition~\ref{prop:m-2} produces a quasipositive braid $\beta$ with $BT=n-2$. If we show that $\beta$ is connected, then we get a 
slice knot, since by Lemma~\ref{chi-sharp} we have $\chi_4(\beta)=1$. To see connectedness, observe that the monodromy of $\beta$ is the product of positive half-twists about 
the arcs $\gamma_1, \gamma_2, \dots \gamma_{n-1}$ constructed in  Proposition~\ref{prop:m-2}. The arc $\gamma_k$ connects the $k$-th puncture to the $(k+1)$-st puncture, so the 
corresponding half-twist acts on the set of punctures by a transposition interchanging the $k$-th and $(k+1)$-st punctures. It follows that the composition of the half-twists
about $\gamma_1, \gamma_2, \dots \gamma_{n-1}$ gives a cyclic permutation of all the punctures: $n \mapsto n-1$, \dots, $3 \mapsto 2$, $2 \mapsto 1$, $1 \mapsto n$. So the resulting 
braid is connected. 
\end{proof}

In general, knots and links can be represented by a braid word of the form
\begin{equation}\label{eq:general-word}
(w_1 \s_{i_1} w_1^{-1})^{\epsilon_1} \ (w_2 \s_{i_2} w_2^{-1})^{\epsilon_2} \cdots (w_m \s_{i_m}w_m^{-1})^{\epsilon_m}
\end{equation}
where the exponents $\epsilon_1,\dots,\epsilon_m \in \{1, -1\}$. 
When $\epsilon_j=1$ the word $w_j \s_{i_j} w_j^{-1}$ is called a positive syllable  
and when  $\epsilon_j=-1$ the word $(w_j \s_{i_j} w_j^{-1})^{-1}$ is called  a negative syllable. 
Let $p_+ \geq 0$ be the number of positive syllables and $p_- \geq 0$ be the number of negative syllables in the above braid representative (\ref{eq:general-word}). 
Thus $p_+ +p_-=m$. 
Note that when $p_+=m$ and $p_-=0$ the braid is quasipositive. 
Theorem~\ref{thm:estimate} has another corollary.

\begin{corollary}\label{corA}
Let $\beta\in B_n$ be the closure of an $n$-braid $$
(w_1 \s_{i_1} w_1^{-1})^{\epsilon_1} \ (w_2 \s_{i_2} w_2^{-1})^{\epsilon_2} \cdots (w_m \s_{i_m}w_m^{-1})^{\epsilon_m}
$$ that does not admit destabilizations and satisfies $\chi_4(\beta)=n-m$. 
Let $p_+$ (resp. $p_-$) be the number of positive (resp. negative) syllables. Then the FDTC of $\beta$ has the following upper and lower bounds:
$$\min\{-p_-+1, 0\} \leq BT(\beta) \leq \max\{p_+-1, 0\}$$
\end{corollary}

\begin{proof}
Given a braid word 
$
(w_1 \s_{i_1} w_1^{-1})^{\epsilon_1} \ (w_2 \s_{i_2} w_2^{-1})^{\epsilon_2} \cdots (w_m \s_{i_m}w_m^{-1})^{\epsilon_m}
$
we define a quasipositive braid word as follows. 
For $j=1,\dots,m$ let 
$$
\delta_j =  \left\{ 
\begin{array}{ccl}
1 & {\rm if } & \epsilon_j=1, \\
0 & {\rm if } & \epsilon_j=-1.
\end{array}
\right.
$$
Let $$\beta_P = (w_1 \s_{i_1} w_1^{-1})^{\delta_1} \ (w_2 \s_{i_2} w_2^{-1})^{\delta_2} \cdots (w_m \s_{i_m}w_m^{-1})^{\delta_m}.
$$
Note that $\beta_P$ is a quasipositive (possibly trivial) braid of qp-length $p_+$.  
For every properly embedded arc $\gamma$ in the $n$-punctured disk $D_n$ the images under $\beta$ and $\beta_P$ satisfy 
$$\beta(\gamma) \preceq \beta_P(\gamma)$$
thus, by Theorem~\ref{thm:estimate} $BT(\beta) \leq BT(\beta_P) \leq \max\{ p_+-1, 0\}$. 

The lower bound can be obtained similarly. 
\end{proof}

Even though Proposition \ref{prop:m-2} claims sharpness, 
the estimates of the FDTC in both Theorem~\ref{thm:estimate} and Corollary~\ref{corA} are in general very rough for most braids as the next Proposition \ref{prop:stab} implies.

\begin{proposition}\label{prop:stab}
Let $\beta \in B_n$. 
If $\hat\beta$ admits a positive (resp. negative) braid destabilization then $0\leq BT(\beta)\leq1$ (resp. $-1\leq BT(\beta)\leq 0$). 
\end{proposition}

\begin{proof}
Let $\hat\beta$ be a closed $n$-braid that admits a positive destabilization. 
Since the FDTC is invariant under conjugation, we may assume that 
$\beta$ is represented by a braid word $w \sigma_{n-1}$ where $w \in B_{n-1}$. 
Define arcs $\delta$ and $\gamma_{n-1}$ as in Figure~\ref{fig:Arc-delta}. 
We may identify the braids $\sigma_{n-1}$ and $w$ with the mapping classes $H_{\gamma_{n-1}}$ and $\phi_w \in \MCG(D_n)$. 
Since $w \in B_{n-1}$ we get $\phi_w(\delta)=\delta$. 
By Lemma~\ref{lemma:fracDehn} we get $BT(\phi_w)=0$.
Likewise, we get $BT(H_{\gamma_{n-1}})=0$ because we can find some properly embedded arc that is fixed by the half twist $H_{\gamma_{n-1}}$.
By Proposition~\ref{prop:FDTCprops} we get $|BT(\beta)| \leq 1$.

Since $T_C^0(\delta)=\delta \prec \phi_\beta(\delta) = H_{\gamma_{n-1}}(\delta)$, Lemma~\ref{lemma:fracDehn} implies that $0\leq BT(\beta)$.
Thus we get $0\leq BT(\beta)\leq1$.
\end{proof}

\begin{figure}[htbp]
\includegraphics*[scale=0.5]{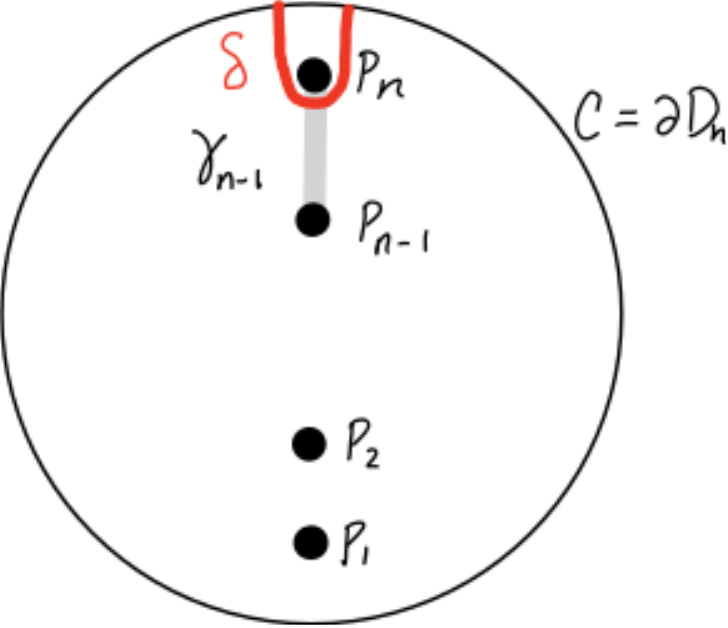} 
\caption{} 
\label{fig:Arc-delta}
\end{figure}

\section{An interesting example} \label{sec:examplesthatkillthedream}

The examples in the last section show that the statement of  Ito's Theorem \ref{thm:ito} does not hold when we replace the genus with the slice genus, since Corollary \ref{slice-BT} produces slice knots that are the closures of $n$-braids with $BT$ equal to $n-2$. In this section we give additional examples, fundamentally different from those in Corollary \ref{slice-BT}, that also show that Theorem \ref{thm:ito} does not hold for the slice genus. Notice that for quasipositive $3$-braids $\beta$ whose closures are slice, Theorem \ref{thm:estimate} implies that $|BT(\beta)| \leq 1$. In Proposition \ref{examplesthatkillthedream}, we produce examples of $3$-braids that have arbitrarily large $|BT|$:

\begin{proposition}\label{examplesthatkillthedream}
For each even $m>0$, there exists a knot $K_m$ and its 3-braid representative $\beta_m$ such that $g_4(K_m)\leq m/2 +1$ while $[\beta_m]_D=m$.
\end{proposition}

\begin{center}
\begin{figure}[h]
\includegraphics[scale=.28]{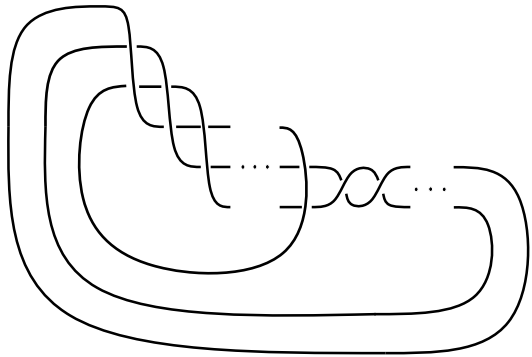} 
\caption{$K_{m,k}$ is the closure of the braid $\beta=(\sigma_2\sigma_1)^{3m+1}\sigma_2^{-2k}$}\label{braid}
\end{figure}
\end{center}

\begin{center}
\begin{figure}
\includegraphics[scale=.4]{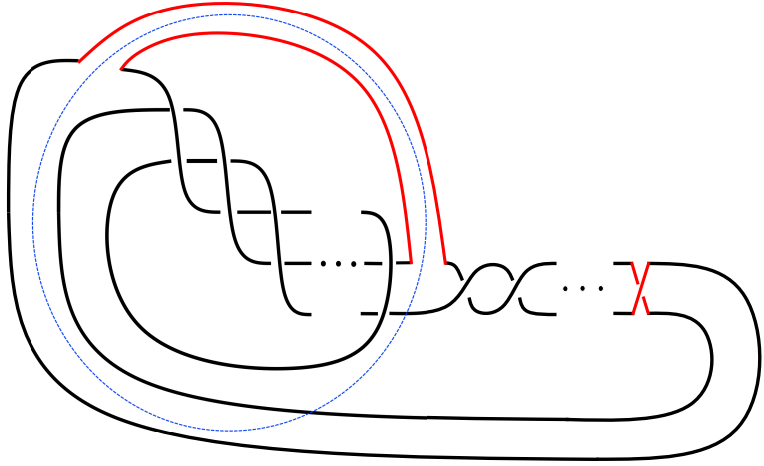} 
\caption{$K_{m,k}$ after two band moves have been performed, with a decomposing sphere}\label{concordance}
\end{figure}
\end{center}

We now describe the examples 
necessary to prove Proposition \ref{examplesthatkillthedream}. Consider the braid word $\beta=(\sigma_2\sigma_1)^{3m+1}\sigma_2^{-2k}$ and let $K_{m,k}=\hat{\beta}$, see Figure \ref{braid}. Note that using the braid relations we can write $\beta=\Delta^{2m}\sigma_2\sigma_1\sigma_2^{-2k}$ where $\Delta=\sigma_1\sigma_2\sigma_1$ is the Garside element for braids of 3 strands.

\begin{proposition}\label{floorofexamplesthatkillthedream} Let $\beta=(\sigma_2\sigma_1)^{3m+1}\sigma_2^{-2k}$. Then we have $[\beta]_D=m$.
\end{proposition}
\begin{proof} To show this, it is enough to show that $\Delta^{2m}\leq_D \beta <_D \Delta^{2m+2}$. To check that $ \beta <_D \Delta^{2m+2}$ we must show that $\beta^{-1}\Delta^{2m+2}$ is $\sigma_i-$positive for $i=1$ or 2. Simplifying, we have 

\begin{align*}
\beta^{-1}\Delta^{2m+2}&=\sigma_2^{2k}\sigma_1^{-1}\sigma_2^{-1}\Delta^{-2m}\Delta^{2m+2}\\
&=\sigma_2^{2k}\sigma_1^{-1}\sigma_2^{-1}\Delta^2\\
&=\sigma_2^{2k}\sigma_1^{-1}\sigma_2^{-1}(\sigma_2\sigma_1)^3\\
&=\sigma_2^{2k}(\sigma_2\sigma_1)^2,
\end{align*}
which is in fact a $\sigma_1$-positive word. On the other hand 
\begin{align*}
(\Delta^{2m})^{-1}\beta&=\Delta^{-2m}\Delta^{2m}\sigma_2\sigma_1\sigma_2^{-2k}\\
&=\sigma_2\sigma_1\sigma_2^{-2k},
\end{align*}
which is also a $\sigma_1$-positive word. The proposition follows.
\end{proof}

We are now ready to prove Proposition \ref{examplesthatkillthedream}. 

\begin{proof}(of Proposition \ref{examplesthatkillthedream})
For any knot in the family $K_{m,k}$ there is a sequence of two band moves on $K$ that yields the connected sum of torus knots $K'=T_{3,3m+1}\#-T_{2,2k+1}$, see Figure \ref{concordance}.

By \cite[Corollary 3]{Fe16}, we have
\[g_4(K')=\mbox{max}\{|\tau(T_{3,3m+1})-\tau(T_{2,2k+1})|,|v(T_{3,3m+1})-v(T_{2,2k+1})|\}\]
where $v(K):=\Upsilon_K(1)$ is the Upsilon invariant evaluated at $t=1$.

The values of $\tau$ and $v$ can be explicitly calculated for some torus knots. Let $p$ and $q$ be positive integers, then by \cite{OS03}
\[\tau(T_{p,q})=\frac{(p-1)(q-1)}{2}.\]
For $v$ we have that 
\[v(T_{2,2k+1})=-k \mbox{ and } v(T_{3,3m+1})=-2m\]
as computed by Feller \cite{Fe16}.

Now for each even integer $m$, let $K_m$ denote the closure of the braid $\beta_m=(\sigma_2\sigma_1)^{3m+1}\sigma_2^{-5m}$. Let $K_m'=T_{3,3m+1}\#-T_{2,5m+1}$ which is obtained from $K_m$ from the two band moves illustrated in Figure \ref{concordance}. Then

\[|\tau(T_{3,3m+1})-\tau(T_{2,5m+1})|=|3m-5m/2|=m/2\]

and

\[|v(T_{3,3m+1})-v(T_{2,5m+1})|=|-2m+5m/2|=m/2.\]

Thus, $g_4(K_m')=m/2$. Since $K_m$ and $K_m'$ are related by two band moves, $g_4(K_m) $ and $g_4(K_m')$ differ by at most 1 and we obtain $g_4(K_m)\leq m/2+1$. Together with Proposition \ref{floorofexamplesthatkillthedream}, this concludes the proof.
\end{proof}

\section{Potential bounds on slice genus from the braid perspective} 
\label{sec:BTgenusbounds}

It is possible that the examples in Proposition \ref{prop:m-2} and Proposition~\ref{examplesthatkillthedream} represent a worst-case scenario. Inspired by these examples and  Theorem~\ref{thm:quasipositiveBTbound}, in this section we further probe the following question.
\quesgenusbounds*
Equivalently, similar questions can be asked about the Dehornoy's floor $[\beta]_D$, since $[\beta]_D$ and the fractional Dehn twist coefficient $BT(\beta)$ are related by~(\ref{BTfloor}).

One of the crucial observations necessary for the examples in Proposition~\ref{examplesthatkillthedream}
is that there exist cobordisms of controlled genus between stacked braids and connected sums of closed braids. 
We observe here that with a positive answer to Question~\ref{conj:genusbounds} (or any similar bounds), this could produce an application: 
one could find bounds on the slice genus of connected sums of large torus knots. 
The slice genus of sums and differences of torus knots has been studied in several papers, see for instance \cite{Fe16, LV18, AA19, Li18, Al19}. 

In addition to the quasipositive braids studied in Section 4, we are able to answer Question~\ref{conj:genusbounds} for certain slice knots. In fact, under some additional hypotheses,
slice knots satisfy stronger bounds. First, note that if $\beta$ closes 
to the unknot, then $|BT(\beta)|\leq 1$ by Theorem~\ref{thm:ito}. (It is easy to find examples showing that even for the unknot, $BT$ does not have to be zero.) 

Further, recall that 
taking the connected sum of a knot and its mirror is a useful method for constructing  slice knots: for any $K$, $K \# -K$ is a slice knot.

\begin{proposition}[Follows directly from \cite{Ma04}] Choose an arbitrary oriented knot $K$ in $S^{3}$. For any braid $\beta$ that closes to 
$\hat\beta = K \# -K$, we have that $|BT(\beta)| \leq 1$. 
\end{proposition} 
\begin{proof}
Theorem 15.3 of \cite{Ma04} states that for any braid $\alpha$ such that $|BT(\alpha)| > 1$ and $\hat\alpha$ is a knot, $\hat\alpha$ is prime - that is, it cannot be expressed as a connected sum. The proposition directly follows by the contrapositive.
\end{proof}

We provide evidence for a positive answer to Question~\ref{conj:genusbounds} for special classes of three-braids in Theorem~\ref{slice3-braids}. 

\thmslicethreebraids*

A powerful tool at our disposal is a classification up to conjugation due to Murasugi. Before proving the theorem, we state the classification of 3-braids and note some properties of slice 3-braids.

\begin{theorem}[\cite{Mu74}]
\label{3-braidclassification}
Let $w$ be a braid word in $B_3$. Then $w$ is conjugate to one of the following:

\begin{enumerate}
\item $\Delta^{2d}\sigma_1\sigma_2^{-a_1}\cdots \sigma_1\sigma_2^{-a_n}$, where the $a_i\geq 0$ with at least one $a_i\neq 0$ 
\item $\Delta^{2d}\sigma_2^m$ for $m\in \mathbb{Z}$
\item $\Delta^{2d}\sigma_1^m\sigma_2^{-1}$ where $m\in\{-1,-2,-3\}$.
\end{enumerate}
\end{theorem}

We will say a braid word in $B_3$ is in Murasugi normal form if it takes one of the above three forms. Baldwin classifies which closures of $3$-braids are quasi-alternating in the following theorem.

\begin{theorem}[Theorem 8.6 in \cite{Bal08}]
\label{QAbraids}
Suppose $L$ is link with braid index at most $3$ and is the closure of a braid $\beta$ represented by $w\in B_3$ which is in Murasugi normal form. Then $L$ is quasi-alternating if and only if one of the the following holds:

\begin{enumerate}
\item $w$ is in the first class and $d\in\{-1,0,1\}$;
\item $w$ is in the second class and either $d=1$ and $m\in\{-1,-2,-3\}$ or $d=-1$ and $m\in \{1,2,3\}$;
\item $w$ is in the third class and $d\in\{0,1\}$.
\end{enumerate}
\end{theorem}

Baldwin also obtains a result on 3-braids with finite concordance order.

\begin{proposition}[Proposition 1.6 in \cite{Bal08}] If $K$ is a knot with braid index at most $3$ and $K$ has finite concordance order, then $K$ can be represented as the closure of a braid of the form $\Delta^{2d}\sigma_1\sigma_2^{-a_1}\cdots \sigma_1\sigma_2^{-a_n}$ where the $a_i\geq 0$ with at least one $a_j\neq 0$, $d\in\{-1,0,1\}$, and the $a_i$ satisfy some further conditions.
\end{proposition}

This implies the following immediate corollary.

\begin{corollary}\label{concordanceQA} A knot $K$ with braid index at most $3$ which has finite concordance order is quasi-alternating.
\end{corollary}

\begin{proposition}\label{3-braidFDTC}
 Let $w\in B_3$ be a braid word in Murasugi normal form. Then

\begin{enumerate}
\item $BT(w)=d$ if $w$ is in class $(1)$.
\item $BT(w)=d$ if $w$ is in class $(2)$.
\item $BT(w)=d-1/3$, $BT(w)=d-1/2$ or $BT(w)=d-2/3$ if $m$ is $-1$, $-2$ or $-3$ respectively, and $w$ is in class $(3)$.
\end{enumerate}
\end{proposition}

\begin{proof}
If $w$ is in class $(1)$ then $BT(w)=d+BT(\sigma_1\sigma_2^{-a_1}\cdots \sigma_1\sigma_2^{-a_n}) $.  If $w$ is in class $(2)$ then $BT(w)=d+BT(\sigma_2^m)=d$.
It follows from Proposition~\ref{arcs-FT} that the second summand in each case is zero,
so we obtain that $BT(\beta)= BT(w)=d$ for $w$ in class $(1)$ or $(2)$.

If  $w$ is in class $(3)$, then $BT(w)=d+BT(\sigma_1^m\sigma_2^{-1})$. Now $BT(\sigma_1^m\sigma_2^{-1})$ is not zero. There are three cases to consider since $m\in \{-1,-2,-3\}$. We compute the following using the Artin relations, and properties of the FDTC enumerated in Proposition \ref{prop:FDTCprops}. 

\begin{enumerate}
\item[(a)] $BT(\sigma_1^{-1}\sigma_2^{-1})=(1/3)BT((\sigma_1^{-1}\sigma_2^{-1})^3)=(1/3)BT(\Delta^{-2})=-1/3$

\item[(b)] $BT(\sigma_1^{-2}\sigma_2^{-1})=(1/2)BT((\sigma_1^{-2}\sigma_2^{-1})^2)=(1/2)BT(\Delta^{-2})=-1/2$

\item[(c)] $BT(\sigma_1^{-3}\sigma_2^{-1})=(1/3)BT((\sigma_1^{-3}\sigma_2^{-1})^3)=(1/3)BT(\Delta^{-4})=-2/3$
\end{enumerate}

Thus for $w$ in class $(3)$ we have that $BT(w)= d-1/3$, $BT(w)=d-1/2$ or $BT(w)=d-2/3$ if $m$ is $-1$, $-2$ or $-3$ respectively.
\end{proof}

\begin{proof}[Proof of Theorem \ref{slice3-braids}]
Corollary \ref{concordanceQA} shows that $\beta$ closes to a quasi-alternating knot. We know that $\beta$ is conjugate to a braid word $w$ in one of the three classes of Theorem \ref{3-braidclassification}. To close to a quasi-alternating knot, $w$ must satisfy the conditions on $d$ in Theorem \ref{QAbraids}. 

If $w$ is in class $(1)$ or $(2)$ this means that $-1\leq BT(w)\leq 1$ by Proposition \ref{3-braidFDTC}. If $w$ is in class $(3)$ then $-2/3\leq BT(w)\leq 1$. The theorem follows from the fact that $BT$ is a conjugacy invariant.
\end{proof}

We conclude the section with some computational evidence toward an affirmative answer to the first bound given in Question \ref{conj:genusbounds}.  We use the Hedden-Mark bound 
of Corollary~\ref{HM-braids}. The quantity $\dim_\mathbb{F} \widehat{HF}(\Sigma(K))$ can be obtained from computations done in \cite{BZ13}.
We set aside  slice knots with 3-braids representatives, as for these Question \ref{conj:genusbounds} is already answered by Theorem \ref{slice3-braids}. 
For all other knots with 12 crossings or fewer, we calculate that the inequality

\begin{equation}
\label{HFbounds?}
\dim_\mathbb{F} \widehat{HF}(\Sigma(K)) - |H_1(\Sigma(K))|+2\leq 2g_4(K)+n-2.
\end{equation}
holds  with 120 exceptions. As before, $n$ stands for the braid index. 

This means that the bound $|BT(\beta)|\leq 2g_4(K)+n-2$ is valid for odd-strand braid representatives all but possibly 120 knots 
with 12 crossings or fewer, ie for at least 96\% of such knots. (According to KnotInfo,  there are 2,977 knots with 12 crossings or fewer.) 

Similarly, we can check that odd-strand braid representatives of all but 5 quasipositive knots of 12 crossings or fewer satisfy inequality~(\ref{HFbounds?}), and thus
they satisfy the first inequality in Question \ref{conj:genusbounds}. (For these knots, this experimental statement is stronger than Theorem~\ref{thm:estimate}, as 
it also applies to non-quasipositive braid representatives of quasipositive knots.)

\section{Fibered knots and knot Floer stable equivalence}\label{sec:knotfloerstable}

In Sections \ref{sec:knotfloerstable} and \ref{sec:fdtcsliceknots}, we collect some observations about the relationship between the fractional Dehn twist coefficient of fibered knots and their topological properties. 

We say that two knots are \emph{knot Floer stably equivalent} if their knot Floer complexes are isomorphic after possibly adding acyclic summands. Hom \cite{Ho17} shows that if two knots in $S^3$ are concordant, then they are knot Floer stably equivalent. Moreover, if two knots are knot Floer stably equivalent, then many invariants derived from knot Floer homology will coincide for the two knots. For example, the Ozsv\'ath-Stipsicz-Szab\'o concordance invariant $\Upsilon_K(t)$ \cite{OSS17} will satisfy $\Upsilon_{K_1}(t) = \Upsilon_{K_2}(t)$ if $K_1$ and $K_2$ are stably equivalent. 

He, Hubbard, and Truong showed in \cite{HHT19} that the $\Upsilon_K(t)$ invariant can detect right-veeringness. More precisely:

\begin{theorem}[{\cite[Theorem 1.3]{HHT19}}]
Suppose $K$ is a fibered knot in $S^{3}$. Then associated to $K$ is an open book decomposition $(\Sigma, \phi)$ of $S^{3}$. If $\Upsilon'_{K}(t) = -g$ for some $t \in [0,1)$, where $g$ is the genus of the fibered surface $\Sigma$, then $\phi : \Sigma \to \Sigma$ is right-veering.
\end{theorem}

We note that Theorem 1.3 in \cite{HHT19} is slightly more general, as it is stated for null-homologous knots in rational homology three-spheres. 
Due to this theorem we are motivated to ask whether the knot Floer stable equivalence class of a fibered knot also can detect right-veeringness.
However, it is easy to see that the answer is no. 

\begin{lemma}
The right-veering property and the FDTC are {\em not} invariants of knot Floer stable equivalence of a fibered knot. 
\end{lemma}

\begin{proof}
Let $K$ be the $(2,1)$-cable of the figure-eight knot, which has positive
FDTC by \cite{KR13} so it is right veering. 
Since $K\#P \#-P$ is concordant to $K$ for any fibered knot $P$, 
the connected sum $K\#P \#-P$ is stably equivalent to $K$. The monodromy of $K\#P \#-P$ sends
some arcs to the right and some to the left. Thus, $K\#P \#-P$ is not right-veering, and the fractional Dehn twist coefficient of $K\#P \#-P$ is zero.
(See Proposition~\ref{arcs-FT} and the discussion that follows it.) 
\end{proof}




In the next section, we will study the fractional Dehn twist coefficient of a special class of fibered knots.

\section{Fractional Dehn twist coefficient of fibered slice knots}\label{sec:fdtcsliceknots}

Many simple examples of fibered slice knots have monodromies with the fractional Dehn twist coefficient zero. For example, given any fibered knot $K$, the slice knot $K \# -  m(K)$ has the FDTC zero. One way to see this is that we can assume there is some arc that the monodromy of $K$ sends to the right, then the corresponding arc is sent to the left by the monodromy for $-m(K)$ and so $K \# -  m(K)$ sends arcs to both the right and the left.

\begin{figure}
  \centering
      \includegraphics[width=0.4\textwidth]{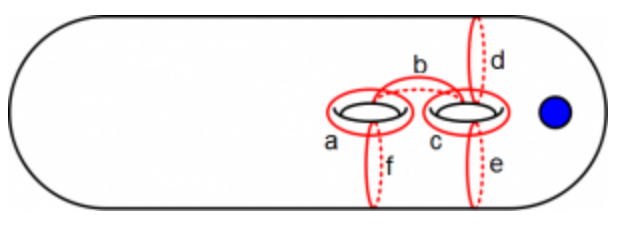}
      \includegraphics[width=0.4\textwidth]{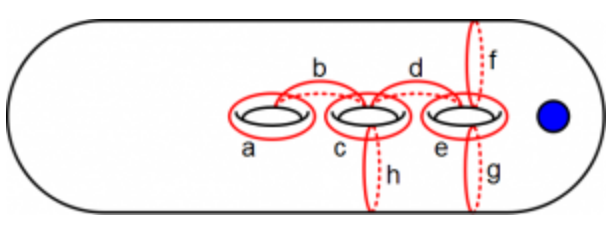}
  \caption{Figures from~\cite{LM20} showing curves whose Dehn twists are generators of the mapping class group of a fiber surface of genus two and three with one puncture.
}
  \label{fig:SurfaceGens}
\end{figure}

There are many more examples of fibered slice knots with fractional Dehn twist coefficient zero. 
For instance, the knot $8_{20}$ is slice and fibered, and has monodromy with fractional Dehn twist coefficient zero \cite[Example 2.8]{KR13}. Figure \ref{fig:fiberedSliceMonodromies} gives several more examples of fibered slice knots with vanishing fractional Dehn twist coefficient $FT = 0$.

\begin{figure}
     \centering
     \begin{subfigure}[b]{0.3\textwidth}
         \centering
         \includegraphics[width=\textwidth]{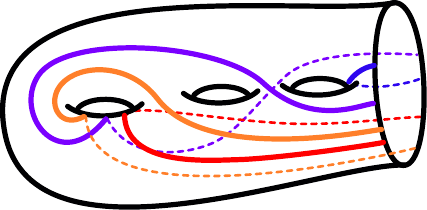}
         \caption{The monodromy $\phi$ for $8_9$ is $abcDEF$}
         \label{fig:8_9}
     \end{subfigure}
     \hfill
     \begin{subfigure}[b]{0.3\textwidth}
         \centering
         \includegraphics[width=\textwidth]{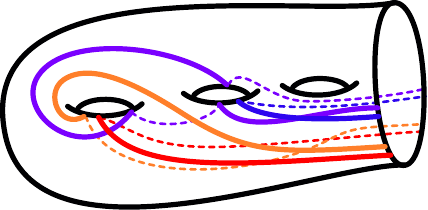}
         \caption{The monodromy $\phi$ for $9_{27}$ is $abCDEf$.}
         \label{fig:9_27}
     \end{subfigure}
     \hfill
     \begin{subfigure}[b]{0.3\textwidth}
         \centering
         \includegraphics[width=\textwidth]{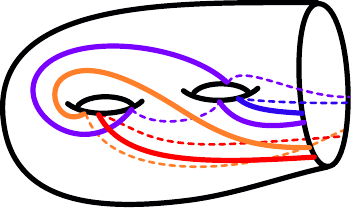}
         \caption{The monodromy $\phi$ for $10_{137}$ is $abCdF$.}
         \label{fig:10_137}
     \end{subfigure}
     \\ \vspace{8mm}
     \begin{subfigure}[b]{0.3\textwidth}
         \centering
         \includegraphics[width=\textwidth]{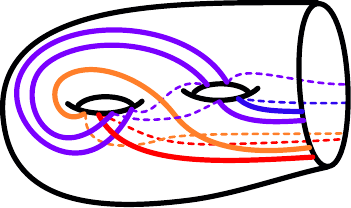}
         \caption{The monodromy $\phi$ for $10_{140}$ is $abbCDf$.}
         \label{fig:}
     \end{subfigure}
     \quad \quad 
     \begin{subfigure}[b]{0.3\textwidth}
         \centering
         \includegraphics[width=\textwidth]{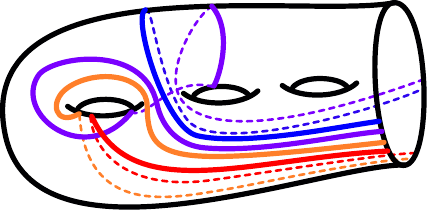} 
         \caption{The monodromy $\phi$ for $11a_{96}$ is $aBcDEf$.}
         \label{fig:11a_96}
     \end{subfigure}
        \caption{Fibered slice knots with vanishing fractional Dehn twist coefficient. Each figure shows a pair of arcs $x$ (in red) and $y$ (in blue) such that $\phi(x)$ (in orange) is to the right of $x$ and $\phi(y)$ (in purple) is to the left of $y$. The monodromies are presented using the conventions of Knot Info \cite{LM20}; see Figure \ref{fig:SurfaceGens}. A lowercase letter means a right-handed Dehn twist and an uppercase letter means a left-handed Dehn twist. A word is read from right to left; thus, $aB$ means first perform a left-handed twist around $b$ then perform a right-handed twist around~$a$.
        }
        \label{fig:fiberedSliceMonodromies}
\end{figure}

Despite the numerous examples of fibered slice knots with vanishing fractional Dehn twist coefficient, the next proposition shows there exist many non-vanishing examples.

\begin{proposition} \label{prop:FTslice}
For every integer $ \vert p \vert \geq 2$, there exist fibered, slice knots with fractional Dehn twist coefficient $\frac{1}{p}$. 
\end{proposition}

\begin{proof}
Our examples are $(p,1)$-cables of fibered, slice knots. Indeed, any cable of a fibered knot $K$ is well-known to be fibered. This follows from \cite{Stallings} or from explicitly building a fibration from the fibration of the companion knot $K$ and the fibration of the pattern torus knot. If $K$ and $K'$ are concordant, then their $(p,q)$-cables are also concordant (see, for example, \cite{HPC}). Thus, if $K$ is a slice knot, then the $(p,1)$-cable of $K$ is also slice. The fractional Dehn twist coefficient of the $(p,1)$-cable of $K$ is $1/p \not = 0$ by \cite[Proposition 4.2]{KR13}.
\end{proof}

The examples in Proposition \ref{prop:FTslice} are cable knots and thus they correspond to reducible monodromies. Restricting to fibered slice knots with pseudo-Anosov monodromies leads to the following question.

\begin{question}\label{sliceFDTC}

If $K$ is a fibered slice knot with pseudo-Anosov monodromy, then is the fractional Dehn twist coefficient zero?

\end{question}

Recall that a result of Gabai (see Theorem~\ref{gabai-thm} stated in the Introduction) establishes a relation between the 3-genus of a fibered knot and the FDTC of the fibration.
A positive answer to Question~\ref{sliceFDTC} would signify a connection between the fractional Dehn twist coefficient and the 4-ball genus.\\

\bibliography{BibliographyWISCON}

\end{document}